\documentclass[a4paper,11pt]{article}

\usepackage[utf8]{inputenc}
\usepackage[T1]{fontenc}
\usepackage{geometry}
\usepackage[francais]{layout}
\usepackage{amsthm}
\usepackage{amsmath}
\usepackage{array,booktabs}
\usepackage{multirow}
\usepackage{dsfont}
\usepackage{array}
\usepackage{comment}
\usepackage{listings}
\usepackage{verbatim}
\usepackage[nottoc, notlof, notlot]{tocbibind}
\usepackage{textcomp}
\usepackage{lmodern}
\usepackage{amssymb}
\usepackage{fancyvrb}
\usepackage{graphicx}
\usepackage{enumerate}
\usepackage{varioref}
\usepackage[dvipsnames]{xcolor}
\usepackage{caption}
\usepackage{wrapfig}
\usepackage{tikz}
\usepackage{hyperref}
\hypersetup{
colorlinks=true,
linkcolor=PineGreen,
urlcolor=MidnightBlue,
citecolor=OliveGreen,
}
\usepackage[capitalise]{cleveref}
\usetikzlibrary{shapes.multipart}
\usetikzlibrary{positioning}
\usetikzlibrary{arrows}
\usepackage{subcaption}
\usepackage{comment}
\usepackage{epsfig}
\usepackage[
backend=biber,
style=alphabetic,
citestyle=authoryear
]{biblatex}
\newtheorem{axiom}{Conjecture}
\newtheorem{remark}{Remark}

\newtheorem{proposition}{Proposition}

\usepackage{comment}
\usepackage{ulem}

\title{Front propagation of a sexual population with evolution of dispersion: a formal analysis
\footnote{European Research Council (ERC) under the European Union’s Horizon 2020 research and innovation program (grant agreement No 639638) and the ANR projects NONLOCAL (ANR-14-CE25-0013) and RESISTE (ANR-18-CE45-0019).}}

\author{L\'eonard Dekens\thanks{Institut Camille Jordan, UMR5208 UCBL/CNRS, Université de Lyon; INRIA, Dracula Team. (\href{mailto:dekens@math.univ-lyon1.fr}{dekens@math.univ-lyon1.fr}).} \and Florian Lavigne \thanks{Aix Marseille Univ, CNRS, Centrale Marseille, I2M, Marseille, France (\href{mailto:florian.lavigne@univ-amu.fr}{florian.lavigne@univ-amu.fr}, \href{mailto:florian.lavigne@inra.fr}{florian.lavigne@inra.fr}).}}

\begin{document}
\maketitle

\begin{abstract}
 The adaptation of biological species to their environment depends on their traits. When various biological processes occur (survival, reproduction, migration, \textit{etc.}), the trait distribution may change with respect to time and space. In the context of invasions, when considering the evolution of a heritable trait that encodes the dispersive ability of individuals, the trait distribution develops a particular spatial structure that leads to the acceleration of the front propagation. That phenomenon is known as spatial sorting. Many biological examples can be cited like the bush cricket in Britain, the cane toad invasion in Australia or the common myna one in South Africa. 
 
 Adopting this framework, recent mathematical studies have led to highlight the influence of the reproductive mode on the front propagation. Asexual populations have been shown to spread with an asymptotic rate of $t^{3/2}$ in a minimal reaction-diffusion model, whereas the analogous rate for sexual populations is of $t^{5/4}$ (where $t$ denotes the time). However, the precise description of the behaviour of the front propagation in the sexual case is still an open question. 
 
 The aim of this paper is to give precise approximations for large times of its position, as well as some features of the local trait distribution at the front. To do so, we solve explicitly the asymptotic problem derived formally. Numerical simulations are shown to confirm these calculations.
\end{abstract}


  {\textbf{AMS}: 35Q92, {\color{black}92D15, 92D25, 35R09, 35B40, 35K57.}}

\section{Introduction}
\label{chap_cemracs_sex:sec:Intro}

Individuals can be more or less adapted to their environment, depending on their traits. Various processes shape the trait distributions: some of them intervene locally, like survival and reproduction, and others highly depend on the spatial structure of the environment, like migration. Bio\-lo\-gi\-cal invasions are an example of a process where the role of space is structuring. As the combination of locally limited amount of resources and large available inhabited space tends to drive individuals further away, the ability to explore can be selected upon. Morphological features can therefore evolve to increase dispersion: closer to the front of the invasion, cane toads in Australia tends to develop longer legs \cite{Phi06}, common myna birds in South Africa and \textit{conocephalus discolor} bush cricket in Britain, larger wings \cite{Berthouly2012,Tho01}.

However, that process is not homogeneous in space: individuals with higher dispersal ability are typically located at the range expansion front. This phenomenon is called \textit{spatial sorting}. Its relationship with the evolution of dispersion has been studied by biologists for the past two decades \cite{Hall17,ShiBro11,Tho01,Tra02,dytham1}. More recently, ma\-the\-ma\-ti\-cal studies have been quantifying its influence on the asymptotic speed of the invasion. {\color{black}Our model equation describes the effects of evolution of a trait $\boldsymbol \theta>1$, which determines the dispersion rate, in space ($\boldsymbol x\in \mathbb R$) and through time ($\boldsymbol t \ge 0$) in a po\-pu\-la\-tion subject to sexual reproduction and competition. The trait density $f(\boldsymbol t, \boldsymbol x,\boldsymbol \theta)$ evolves according to}:
\begin{equation}
  \partial_{\boldsymbol{t}} \boldsymbol{f}(\boldsymbol{t},\boldsymbol{x},\boldsymbol{\theta})= \boldsymbol{r} \, [\,\underbrace{\boldsymbol{B}[\boldsymbol{f}](\boldsymbol{t},\boldsymbol{x},\boldsymbol{\theta})}_{\text{reproduction}}-\underbrace{ \boldsymbol{K}^{-1}\boldsymbol{\varrho}(\boldsymbol{t}, \boldsymbol{x})\boldsymbol{f}(\boldsymbol{t},\boldsymbol{x},\boldsymbol{\theta})}_{\text{competition}}\,]+ \underbrace{\boldsymbol{\theta} \Delta_{\boldsymbol{x}} \boldsymbol{f}(\boldsymbol{t}, \boldsymbol{x}, \boldsymbol{\theta})}_{\text{dispersion}},
  \label{chap_cemracs_sex:eq:generalintro}
\end{equation}
{\color{black}for $\boldsymbol {\Delta_x} $ the Laplace operator with respect to $\boldsymbol x$.}

When the dispersal rate is possibly unbounded, the relationship between the front propagation and sustained spatial sorting leads to an acceleration of front propagation \cite{Ber15,Bouin_Calvez_Meunier_Mirrahimi_Perthame_Raoul_Voituriez_2012,Bou17,calvez2018nonlocal}, contrary to the case of constant dispersion for which it is well established that the front expands asymptotically at constant speed \cite{AroWei78, BerHamNad08, Fang11, Gen06, Gou00, HamRyz13, Mirrahimi_Raoul_2013}.

To our knowledge, analytical results describing the asymptotic accelerating rate of propagation exist only for asexual (clonal) populations (\textit{e.g.}, see \cite{Ber15,Bou17,calvez2018nonlocal}), for {\color{black}which} the reproduction operator in \eqref{chap_cemracs_sex:eq:generalintro} is :
\[\boldsymbol{B}[\boldsymbol{f}] = \boldsymbol{f}+\boldsymbol{\sigma^2\Delta_\theta \,f},\]
for some constant $\boldsymbol{\sigma^2}\ge 0$ depending on the mutation variance and mutation rate {\color{black} and for $\boldsymbol {\Delta_\theta} $ the Laplace operator with respect to $\boldsymbol \theta$}. In this case, the position of the population range asymptotically expands as $t^{3/2}$ (see \cite{Ber15,Bouin_Calvez_Meunier_Mirrahimi_Perthame_Raoul_Voituriez_2012,Bou17,calvez2018nonlocal} for more details). Furthermore, the precise asymptotic position of the front has been derived in \cite{calvez2018nonlocal}, by specifying the prefactor term. The value of this prefactor is sensitive to how the competition is modelled : when it is local in trait, it has been shown to be equal to a larger value \cite{Ber15,Bouin_Calvez_Meunier_Mirrahimi_Perthame_Raoul_Voituriez_2012,Bou17}. 

However, as reproductive mode is thought to potentially significantly influence the rate of propagation \cite{Williams_Hufbauer_Miller_2019}, we take interest into invasions of sexually re\-pro\-du\-cing populations. An analogous model as for asexual populations can be built u\-sing Fisher's in\-fi\-ni\-te\-si\-mal model, a model of allelic segregation that has been studied and used for a century in quantitative genetics, a branch of evolutionary biology  \cite{Barton_Etheridge_Veber_2017,Bulmer_1972,Fisher,Lange_1978,Tuf00,Turelli_2017,Turelli_Bartont}. This model has also been used to model sexually repro\-du\-cing populations in several integro-differential studies \cite{Boui17,Calvez_Garnier_Patout_2019,Mirrahimi_Raoul_2013,Raoul_2017}, with the following reproduction operator in \eqref{chap_cemracs_sex:eq:generalintro}:
\newcommand{\thmin}{\theta_{\min}}
\newcommand{\R}{\mathbb R}
\newcommand{\md}{\mathrm{d}}
\begin{equation*}
\boldsymbol{B}[\boldsymbol{f}](\boldsymbol{t},\boldsymbol{x},\boldsymbol{\theta})
= \iint_{(\boldsymbol{\theta_{\min}},\infty)^2} \mathcal{G}_{\boldsymbol{\lambda}}\left[ \boldsymbol{\theta} - \frac{\boldsymbol{\theta}_1+\boldsymbol{\theta}_2 }{2}\right] \, \frac{\boldsymbol{f}(\boldsymbol{t}, \boldsymbol{x}, \boldsymbol{\theta}_1) \, {\boldsymbol{f}(\boldsymbol{t}, \boldsymbol{x}, \boldsymbol{\theta}_2)}}{\boldsymbol{\varrho}(\boldsymbol{t},\boldsymbol{x})} \md\boldsymbol{\theta}_1\, \md\boldsymbol{\theta}_2.
\end{equation*}

It assumes that the trait of the offspring is given by the mean parental trait up to a random normal deviation given by $\mathcal{G}_{\boldsymbol{\lambda}}$ with constant segregational variance $\boldsymbol{\lambda^2}$. Using this model, the authors of the report \cite{proceeding} predicted and numerically confirmed an asymptotic invasion rate of $t^{5/4}$ for sexually reproducing populations.

However, to understand the complexity of the interplay between ecology and evolution in the dynamics of an invasion, the relationship between the propagation and the trait distribution has to be untangled. That requires to describe precisely the trait distribution and the effect of spatial sorting at the front of the invasion, which is the goal of this paper. First, we present our model and the explicit formula that we derive to approximate the position of the front propagation and its local trait distribution at large times (\cref{chap_cemracs_sex:sec:Modelandresult}). Next, we present numerical simulations that confirm this formula (\cref{chap_cemracs_sex:sec:numeric}). Finally, we derive formally the limit problem for large times and find an explicit solution to it (\cref{chap_cemracs_sex:sec:proof}).

\section{Deterministic model}
\label{chap_cemracs_sex:sec:Modelandresult}

In this section, we present the integro - differential model that we use and state our formal result as an approximation of the solutions of the resulting equation. The population is described according to its location $\boldsymbol{x} \in \R$ and its dispersive trait $\boldsymbol{\theta} \in (\boldsymbol{\thmin},+ \infty)$, with $\boldsymbol{\thmin}>0$. Here we are interested by the evolution of the density $\boldsymbol{f}(\boldsymbol{t},\boldsymbol{x},\boldsymbol{\theta})$ of individuals being at time $\boldsymbol{t}\ge 0$ at the location $\boldsymbol{x}\in \R$, presenting the trait $\boldsymbol{\theta}$. We also assume that, initially, the density is compactly supported.

\textbf{Our model.} The evolution of the density $\boldsymbol{f}(\boldsymbol{t},\boldsymbol{x},\boldsymbol{\theta})$ can be modeled with the following reaction - diffusion equation for all $\boldsymbol{t}>0$, $\boldsymbol{x}\in\mathbb R$ and $\boldsymbol{\theta}>\boldsymbol{\theta_{\min}}$:
\begin{equation}
  \partial_{\boldsymbol{t}} \boldsymbol{f}(\boldsymbol{t},\boldsymbol{x},\boldsymbol{\theta})= \boldsymbol{r} \, \left[\boldsymbol{B}[\boldsymbol{f}](\boldsymbol{t},\boldsymbol{x},\boldsymbol{\theta}) - \boldsymbol{K}^{-1}\boldsymbol{\varrho}(\boldsymbol{t}, \boldsymbol{x})\boldsymbol{f}(\boldsymbol{t},\boldsymbol{x},\boldsymbol{\theta})\right]+ \boldsymbol{\theta} \Delta_{\boldsymbol{x}} \boldsymbol{f}(\boldsymbol{t}, \boldsymbol{x}, \boldsymbol{\theta}),
    \label{chap_cemracs_sex:eq:general}
\end{equation}
where $\boldsymbol{r}>0$ and $\boldsymbol{K}>0$ are fixed constants, and $\boldsymbol{\varrho}(\boldsymbol{t},\boldsymbol{x})\,:=\,\int_{\boldsymbol{\thmin}}^\infty \boldsymbol{f}(\boldsymbol{t},\boldsymbol{x},\boldsymbol{\theta})\,\md\boldsymbol{\theta}$ is the population size at $\boldsymbol{x}\in\R$ and time $\boldsymbol{t}>0$. We will detail the reaction term $\boldsymbol{B}[\boldsymbol{f}]$ later. At first, let us discuss the modelling motivation of each term.

First, the term $\boldsymbol{r} \, \left[\boldsymbol{B}[\boldsymbol{f}](\boldsymbol{t},\boldsymbol{x},\boldsymbol{\theta}) - \boldsymbol{K}^{-1}\boldsymbol{\varrho}(\boldsymbol{t}, \boldsymbol{x})\boldsymbol{f}(\boldsymbol{t},\boldsymbol{x},\boldsymbol{\theta})\right]$ is analogous to a logistic growth term that models \textit{reproduction} and \textit{competition}. More precisely, the reproduction term $\boldsymbol{B}[\boldsymbol{f}](\boldsymbol{t},\boldsymbol{x},\boldsymbol{\theta})$ represents the number of new individuals that are born with the trait $\boldsymbol{\theta}$ at time $\boldsymbol{t}\geq 0$ and position $\boldsymbol{x}\in\R$ and we will detail the modelling of the segregational process later. Moreover, at point $\boldsymbol{x}\in \R$ and at time $\boldsymbol{t}\geq 0$, there is a competition between individuals for resources, related to the parameter $\boldsymbol K$ which is a measure of the \textit{{carrying} capacity} of the environment. When the local population size at $\boldsymbol{x}$ is relatively small -- $\boldsymbol{\varrho}(\boldsymbol{t}, \boldsymbol{x})\ll\boldsymbol{K}$ -- the local population disposes of enough resources to allow an exponential - like growth, while, if $\boldsymbol{\varrho}(\boldsymbol{t},\boldsymbol{x})\gg \boldsymbol{K}$, then competition between individuals is strong, and consequently the local population size decreases. The constant $\boldsymbol{r}>0$ is therefore called \textit{growth rate at low density}.

Then, the diffusion term $\boldsymbol{\theta}\,\Delta_{\boldsymbol{x}} \boldsymbol{f}$ models the \textit{dispersion} phenomenon. Individuals are assumed to diffuse through space at each time $t$, at a rate given by the dispersive trait $\boldsymbol{\theta}\ge \boldsymbol{\thmin}$. When $\boldsymbol{\theta}$ gets larger, it models situations like having longer legs or bigger wings, which potentially give an advantage to explore a new environment faster. 

Finally, let us come back to the reproduction operator $\boldsymbol{B}[\boldsymbol{f}]$. We consider a monoecious population in which the individuals breed randomly and only with those at the same location $\boldsymbol{x}\in\R$. At time $\boldsymbol{t}$, an individual with trait $\boldsymbol{\theta}_1$ finds a mate with trait $\boldsymbol{\theta}_2$ with the probability density equal to the trait frequency at position $\boldsymbol{x}:$ $\boldsymbol{f}(\boldsymbol{t}, \boldsymbol{x}, \boldsymbol{\theta}_2) /\boldsymbol{\varrho}(\boldsymbol{t}, \boldsymbol{x})$. To model the segregation, we use Fisher's infinitesimal model, which classically states that the offspring trait differs from the mean parental trait $(\boldsymbol{\theta}_1+\boldsymbol{\theta}_2)/ 2$ according to a normal distribution with a segregrational variance $ \boldsymbol{\lambda}^2>0$ assumed to be constant and independent of the parental trait values. These assumptions imply the following formulation of the reproduction term:
\begin{equation*}
\boldsymbol{B}[\boldsymbol{f}](\boldsymbol{t},\boldsymbol{x},\boldsymbol{\theta})
= \iint_{(\boldsymbol{\thmin},\infty)^2} \mathcal{G}_{\boldsymbol{\lambda}}\left[ \boldsymbol{\theta} - {\boldsymbol{\theta}_1+\boldsymbol{\theta}_2 \over 2}\right] \, {\boldsymbol{f}(\boldsymbol{t}, \boldsymbol{x}, \boldsymbol{\theta}_1) \, {\boldsymbol{f}(\boldsymbol{t}, \boldsymbol{x}, \boldsymbol{\theta}_2)} \over \boldsymbol{\varrho}(\boldsymbol{t},\boldsymbol{x})} \md\boldsymbol{\theta}_1\, \md\boldsymbol{\theta}_2.
\end{equation*}
The term $\mathcal{G}_{\boldsymbol{\lambda}} [ \boldsymbol{\theta} - {(\boldsymbol{\theta}_1+\boldsymbol{\theta}_2) / 2}]$, symbolizing the stochasticity of the segregation process, is defined as a normalized Gaussian density with variance $\boldsymbol{\lambda}^2>0$, that is:
\begin{equation}\label{chap_cemracs_sex:def:gauss}
\mathcal{G}_{\boldsymbol{\lambda}}(\boldsymbol{\theta}) := {1\over \sqrt{2\pi \boldsymbol{\lambda}^2}} \, \exp\left[ - \ {\boldsymbol{\theta}^2\over 2\boldsymbol{\lambda}^2}\right] .
\end{equation}

Let us rescale the equation by setting :
\[t = \boldsymbol{r}\boldsymbol{t},\qquad x = \sqrt{\frac{\boldsymbol{r}}{\boldsymbol{\thmin}}}\, \boldsymbol{x},\qquad \theta =\, \frac{\boldsymbol{\theta}}{\boldsymbol{\thmin}}\,,\quad \hbox{and} \quad f(t,x,\theta) = \frac{\boldsymbol{\thmin}}{\boldsymbol{K}}\boldsymbol{f}(\boldsymbol{t},\boldsymbol{x},\boldsymbol{\theta}).\]
Then, we can simplify the previous PDE into:
\begin{equation}
\partial_t f(t, x, \theta) = B[f](t,x,\theta) - \varrho(t, x) f(t, x, \theta)+\theta \Delta_x f(t, x, \theta),
\label{chap_cemracs_sex:eq:PDE_apres_simplification_et_variance_generale}
\end{equation}
with the rescaled population size:
\[
\varrho(t,x) = \int_1^\infty f(t,x,\theta)\, \md \theta.
\]
By this simplification, the reproduction term is:
\begin{equation}
  B[f](t,x,\theta)=\iint_{(1,\infty)^2} \mathcal G_{ \lambda} \left[ \theta - {\theta_1+\theta_2 \over 2}\right] \, f(t, x, \theta_1) \, {f(t, x, \theta_2) \over \varrho(t,x)} \, \md\theta_1\, \md\theta_2,
\end{equation}
where $\mathcal G_\lambda$ is given by \eqref{chap_cemracs_sex:def:gauss}, and $\lambda = \boldsymbol{\lambda} /\boldsymbol{\thmin}$. One can notice the {\color{black}truncation} at the bottom level $\theta_{\min} = 1$, chosen for the sake of simplicity (note that $\theta_{\min}$ can only take positive values), which does not influence the long time asymptotics in the subsequent analysis as $\theta$ is expected to take large values at the front.

\textbf{Main result.} In this paper, we denote by $x\cdot J$, for some  $x\in \mathbb R$ and $J=[a,b]$, the interval $[xa,xb]$ and $|J|$ the length of the interval $J$. As some computations are only formal, we state our main result as a conjecture:

\begin{axiom}\label{chap_cemracs_sex:thm:main}
Define the constant
\begin{equation}
  y_c=4\left( {\lambda \over 3} \right)^{1/2} .
  \label{chap_cemracs_sex:form:yc}
\end{equation}
{There exists {\color{black}an interval of trait values} $J_0$ centered in 1 such that, for all $J\subset J_0$ open interval centered in 1, the density $f$ at large time $t\ge 0$ can be approximated by}:\small
\[
f(t,x,\theta) = \left\{ \begin{array}{l}
  \!\exp \left[ - \ {1\over 4\lambda ^2} \ \left[\theta - \lambda^{4/5}(6x^2)^{1/5}\right]^2    +\underset {t\to \infty}{\mathcal{O}}({|J|^2})\right] , \\ \\ \quad\qquad\qquad\qquad\qquad\qquad\qquad\qquad\quad \text{ for } x\leq y_c\ t^{5/4},\; \theta \in {\lambda^{4/5}(6x^2)^{1/5}}\cdot J, \\ \\
  {\color{black}\!\exp\left[\left(1- \left(\frac{x}{y_c\,t^{5/4}}\right)^{4/3}\right)t\right]}  \exp\left[ - {1\over 4\lambda^2}{\left[\theta - \left({3\lambda^2 x^2\over 2t}\right)^{1/3} \right]^2}\!\!\!+ \!\underset {t\to \infty}{\mathcal{O}}\!\left({|J|^2\frac{x^{8/3}}{t^{10/3}}}\right)\right] , \\
  \\
  \quad\qquad\qquad\qquad\qquad\qquad\qquad\qquad\quad \text{ for } x\geq y_c\ t^{5/4},\; \theta \in {\left({3\lambda^2 x^2\over 2t}\right)^{1/3}}\cdot J.
\end{array}\right.
\]\normalsize

{\color{black}For $x\ge y_c\,t^{5/4}$, we call the coefficient $\exp\left[\left(1- \left(\frac{x}{y_c\,t^{5/4}}\right)^{1/3}\right)t\right]$ the prefactor of the trait distribution, which is of the form $\exp\left[-c\left(\frac{x}{y_c\,t^{5/4}}\right)\,t\right]$, where the function $c$ is positive and increasing on $]1,+\infty[$.}
\end{axiom}

\noindent The justification of this conjecture is postponed to \cref{chap_cemracs_sex:sec:proof}.

\cref{chap_cemracs_sex:thm:main} yields that at each time $t\geq 0$ large enough, the {\color{black} propagating front} is at the position:
\begin{equation}
  X(t) \approx y_c \, t^{5/4} = 4\left( { \lambda \over 3} \right)^{1/2} \ t^{5/4}.
  \label{chap_cemracs_sex:pos_front}
\end{equation}

Additionally, at large time $t$ and all space position $x\in\R$, {\color{black}the dispersive trait is normally distributed, with }variance $2\lambda^2$. Behind the front, \textit{i.e.}, at all position $x\ll X(t)$, the {\color{black}mean of the dispersive trait} $\overline{\theta}$ can be approximated by the value:

\begin{equation}
\label{meanthetabehind}
\overline{\theta}(x)\approx \lambda^{4/5}(6x^2)^{1/5},
\end{equation}
while ahead of the front, \textit{i.e.}, at all position $x\gg X(t)$, it can be approximated by:
\begin{equation}
\label{meanthetaahead}
\overline{\theta}(t,x)\approx \left({3\lambda^2 x^2\over 2t}\right)^{1/3}.
\end{equation}

{\color{black}Moreover, the prefactor of the distribution trait, $\exp\left[-c\left(\frac{y}{y_c}\right)\,t\right]$, with $c>0$ increasing on $]1,+\infty[$ and $y= {t^{-5/4}}x$, indicates that, ahead of the front, the population size pre\-su\-ma\-bly decreases with regard to the rescaled space variable $y$ at a given time $t>0$.}

\section{Simulations and validation}
\label{chap_cemracs_sex:sec:numeric}

In this section, we display numerical simulations, in order to validate the approximation of the solution of the Eq. \eqref{chap_cemracs_sex:eq:PDE_apres_simplification_et_variance_generale} provided by \cref{chap_cemracs_sex:thm:main}. The initial distribution used for simulation is assumed to be a truncated Gaussian distribution:
\begin{equation}
f(0, x, \theta) = {\sqrt{2\over \pi}}\, \exp\left[ - \, {x^2+(1 - \theta)^2 \over 2} \right]\, \mathds 1_{\theta \geq 1},
\label{chap_cemracs_sex:eq:CI}
\end{equation}
with $\mathds 1_{\theta \geq 1}$ the characteristic function of $\{\theta\geq 1\}$. The segregational variance $\lambda^2$ is taken equal to $1/2$. The discretization of the sexual reproduction term $B[f]$ represents the biggest challenge for the simulations, in comparison to the asexual case (see \cite{proceeding}). 

\subsection{Scheme}

We consider $x_{\max}\geq 0$ and $\theta_{\max}\geq 1$ so that we work with {\color{black}tuples} $(x,\theta)$ in the bounded domain $[0,x_{\max}]\times[1,\theta_{\max}]$, discretized with the meshes $(x_i)_{1\leq i \leq N_x}$ and $(\theta_j)_{1\leq j \leq N_\theta}$, respectively of {\color{black}step length} $\delta x>0$ and $\delta \theta>0$. As for the time discretization, let $\delta t>0$ be a time {\color{black}step length}, and let us define for all $n\in\mathbb{N}$, $t_n\,:=\,n\,\delta t$. We denote by $A^N_x$ the matrix of the discrete Laplace operator in $x$ of size $N_x$ with Neumann boundary condition at $x=0$ and Dirichlet boundary condition at $x=x_{\max}$:
\[
A^N_x = \dfrac{1}{\delta x^2}\,\left( \begin{matrix}
 - 1 & 1& & & (0)\\
1 & - 2 & 1 \\
& \ddots & \ddots & \ddots \\
& & 1 & - 2 & 1 \\
(0)&&& 1& - 2
\end{matrix}
\right)\in \mathcal M_{N_x}(\R),
\]
and the diagonal matrix:
\[
D_{\theta} = \left( \begin{matrix} 
\theta_1 & & (0)\\
& \ddots & \\
(0) & & \theta_{N_\theta}
\end{matrix} \right)\in \mathcal M _{N_\theta}(\R).
\]
Futhermore, we introduce a 3D hypermatrix $G_\theta\in M_{N_\theta,N_\theta,N_\theta}(\R)$ such that:
\[
\forall i,j,k, \, G_\theta(i,j,k) = \mathcal G_\lambda \left[ \theta_k - {\theta_i+\theta_j\over 2}\right],
\]
representing the discretization of the segregation kernel ($\mathcal G_\lambda$ given by \eqref{chap_cemracs_sex:def:gauss}).

For all $n\in\mathbb{N}$, we approximate $\left(f(t_n,x_i,\theta_j)\right)_{1\leq i \leq N_x,1 \leq j \leq N_\theta}$ {\color{black}by} a matrix: 
\[
F^n\,=\,\left(F^n_{ij}\right)_{1\leq i \leq N_x,1 \leq j \leq N_\theta} \in \mathcal M_{N_x,N_\theta}(\R),
\]
and the population size $(\varrho(t_n,x_i))_{1\le i \le N_x}$ {\color{black}by} the vector:
\[
\widetilde \varrho^n_i := \sum_{k=1}^{N_\theta} F^n_{i,k}\, \delta \theta \approx \varrho(t_n,x_i),
\]
using the following scheme. At each time iteration $n$,
\begin{enumerate}
  \item For {\color{black}every} index $1\le k\le N_\theta$, we compute the vector $V^{n}_{k,l}$ defined by:
  \[
  \forall  l,\, V^{n}_{k,l}:= \delta\theta^2\, \left[F^n\,G_\theta(\cdot,\cdot, k)\, (F^n)^T\right]_{ll}.
  \]
We can check that $V^{n}_{k,l}$ is the discretization of the reproduction integral term:
  \begin{align*}
  V^{n}_{k,l} & = \delta\theta^2\, \sum_{i,j=1}^{N_\theta} F^n_{l,i} G_\theta(i,j,k)F^n_{l,j},\\
  & \approx \delta\theta^2\, \sum_{i,j=1}^{N_\theta} f(t_n,x_l,\theta_i) \mathcal G_\lambda\left[\theta_k - {\theta_i+\theta_j\over 2}\right]f(t_n,x_l,\theta_j),\\
  & \approx \iint_{(1,\infty)^2} f(t_n,x_l,\theta_1) \mathcal G_\lambda\left[\theta_k - {\theta_1+\theta_2\over 2}\right]f(t_n,x_l,\theta_2)\,\md\theta_1\md\theta_2.
  \end{align*}
  Now to compute the reproduction matrix $\text{Mat}_{\text{Reprod}}\in \mathcal M_{N_x,N_\theta}(\R)$, we need to divide the previous quantities by the corresponding $\widetilde \varrho_i^n$. To be consistent, we set:
  \[
  \forall i, k,\, \text{Mat}^n_{\text{Reprod}}(i,k)=\left\{
  \begin{array}{ll}
    V^{n}_{k,i}/\widetilde \varrho^n_i, & \text{ if } \widetilde \varrho^n_i>0, \\
    0, & \text{else.}
  \end{array}\right.
  \]
  \item We define the diagonal matrix $D_\varrho^n:=\text{diag}\left((\widetilde \varrho^n_i)_{1\le i \le N_x}\right)\in \mathcal M_{N_x}(\R)$.
  \item We approximate in time using an explicit Euler scheme that is for all $n\in \mathbb N$:
  \begin{equation}
    F^{n+1}:= F^n +\delta t\left[ A_x^N\times F^n\times D_\theta + r\, \left( \text{Mat}^n_{\text{Reprod}} - K^{ - 1} \times D_\varrho^n\times F^n\right)\right].
    \label{chap_cemracs_sex:scheme_eq}
  \end{equation}
  In this section, the parameters $r$ and $K$ are equal to 1. The general scheme~\ref{chap_cemracs_sex:scheme_eq} is used in supplementary materials, to show the effects of differents parameters on the invasion.
\end{enumerate}

To be sure that this scheme gives a good approximation of the solution of the PDE~\eqref{chap_cemracs_sex:eq:PDE_apres_simplification_et_variance_generale}, the spatial step $\delta x$ is taken large enough.

\subsection{Numerical results} We show our results of simulations of the solution of the Eq.~\eqref{chap_cemracs_sex:eq:PDE_apres_simplification_et_variance_generale} in two figures Fig.~\ref{chap_cemracs_sex:fig1} and Fig.~\ref{chap_cemracs_sex:fig2}. In the first one, we display different features of the front, whereas in the second one, we compare the numerical trait distribution behind the front with the approximation formally obtained in \cref{chap_cemracs_sex:thm:main}.

\begin{figure}
  \centering
  \subfloat[]{\includegraphics[scale=0.6]{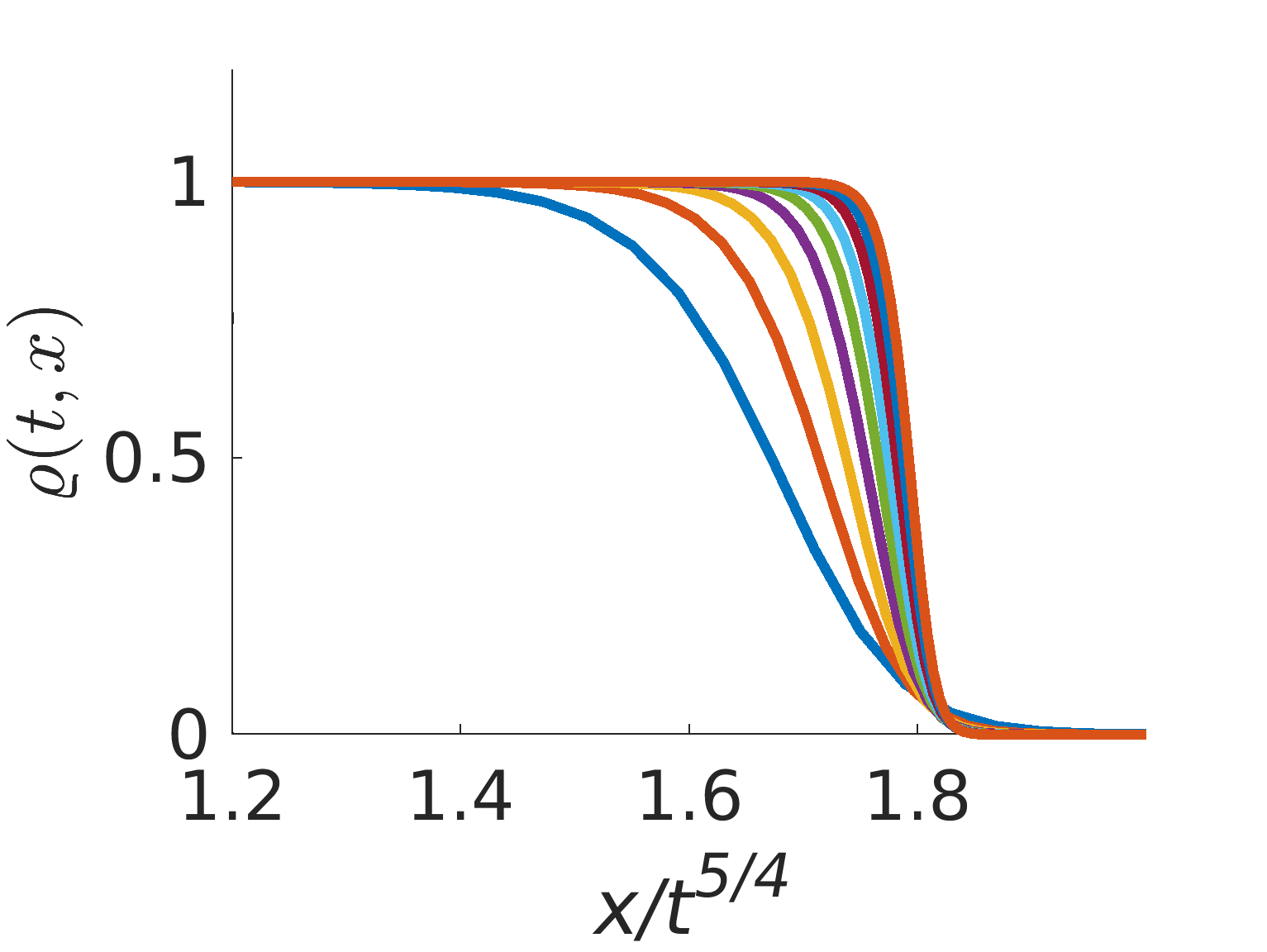}}
\\  \subfloat[]{\includegraphics[scale=0.6]{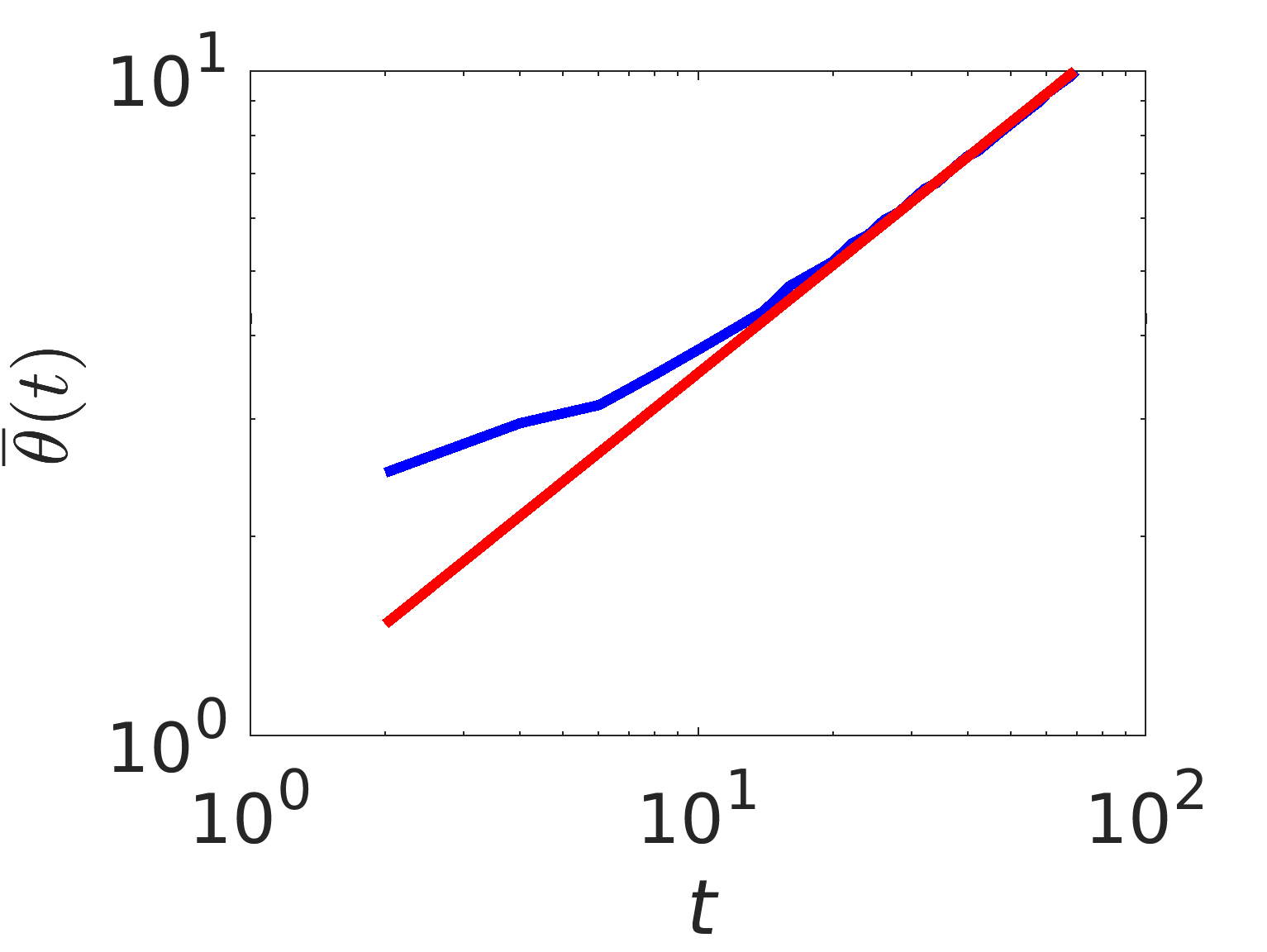}} \hspace{0.3cm}

  \caption[Simulations of the invasion of a sexual population]{\textbf{ Simulations of the invasion of a sexual population}, associated to the Eq.\eqref{chap_cemracs_sex:eq:PDE_apres_simplification_et_variance_generale} with parameters $\delta t = 0.02$, $\delta x = 4$, $\delta \theta =2 /3$, $x_{\max} = 3000$ and $\theta_{\max} = 201$. (a) Plot of the {\color{black}population size} $\varrho(t,\cdot)$ for successive fixed times at regular intervals from $t=20$ to $t=200$, with respect to the auto - similar variable $x t^{ - 5/4}$. (b) Plot of the {\color{black}mean of the dispersive trait} $\overline{\theta}^{num}(t)$ (see \eqref{chap_cemracs_sex:eq:theta(t)_num}) at the front position with respect to time (blue curve) and of the function $t \to {\color{black}1.02} t^{{\color{black}0.54}}$ (red curve), in $\log - \log$ scale.}
  \label{chap_cemracs_sex:fig1}
\end{figure}

    In the top subfigure Fig.~\ref{chap_cemracs_sex:fig1} (a), the population size $\varrho(t,x)$ is displayed at multiple time regularly spaced between $t=20$ and $t=200$ for different scaled position $x$. As expected, thanks to Fig.~\ref{chap_cemracs_sex:fig1} (a), we can see that this front accelerates: there exists a constant $y_c^{num}$ such that the front at time $t$ is at position:
    \begin{equation*}
      X^{num}(t) = y_c^{num} \, t^{5/4},
    \end{equation*}
    where the numerical front position $X^{num}(t)$ at time $t\ge 0$ is defined by:
    \begin{equation}
    X^{num}(t_n)=x_{i^{num}(t_n)}, \quad \hbox{with} \quad i^{num}(t_n)\,:=\,\underset{1\le i \le N_x }{\text{argmin}}\,\left| \widetilde\varrho^n_i - 0.01 \right|.
    \label{chap_cemracs_sex:eq:X(t)_num}
    \end{equation}
    {More precisely, thanks to a linear regression, the constant $y_c^{num}$ can be approximated by {\color{black}2.1}, {\color{black}and the exponent of $t$} by {\color{black}1.22} (with $R^2=1$ and $p$-value $<10^{-4}$).}
    These numerical results are consistent with 
    \eqref{chap_cemracs_sex:pos_front}, which numerically gives:
    \[
    X(t) = 4\left(1\over {2\times 9} \right) ^{1/4} t^{5/4} \approx 1.94\, t^{5/4}.
    \]
    With Fig.~\ref{chap_cemracs_sex:fig1} (b), we confirm that the {\color{black}mean of the dispersive trait} at the front that we get from the numerical simulations is {\color{black}quite} consistent with the approximation given by \cref{chap_cemracs_sex:thm:main}. Precisely, let us define the {\color{black}mean of the dispersive trait} $\bar\theta^{num} (t)$ at the front position $X^{num}(t)$, given by:
    \begin{equation}
    \bar{\theta}^{num}(t) \,:=\, \dfrac{\int_\R\,\theta\,f(t,X^{num}(t),\theta)\,\md\theta}{\varrho(t,X^{num}(t))} .
    \label{chap_cemracs_sex:eq:theta(t)_num}
    \end{equation}

\begin{figure}[h!]
  \centering
  \subfloat[]{\includegraphics[scale=0.6]{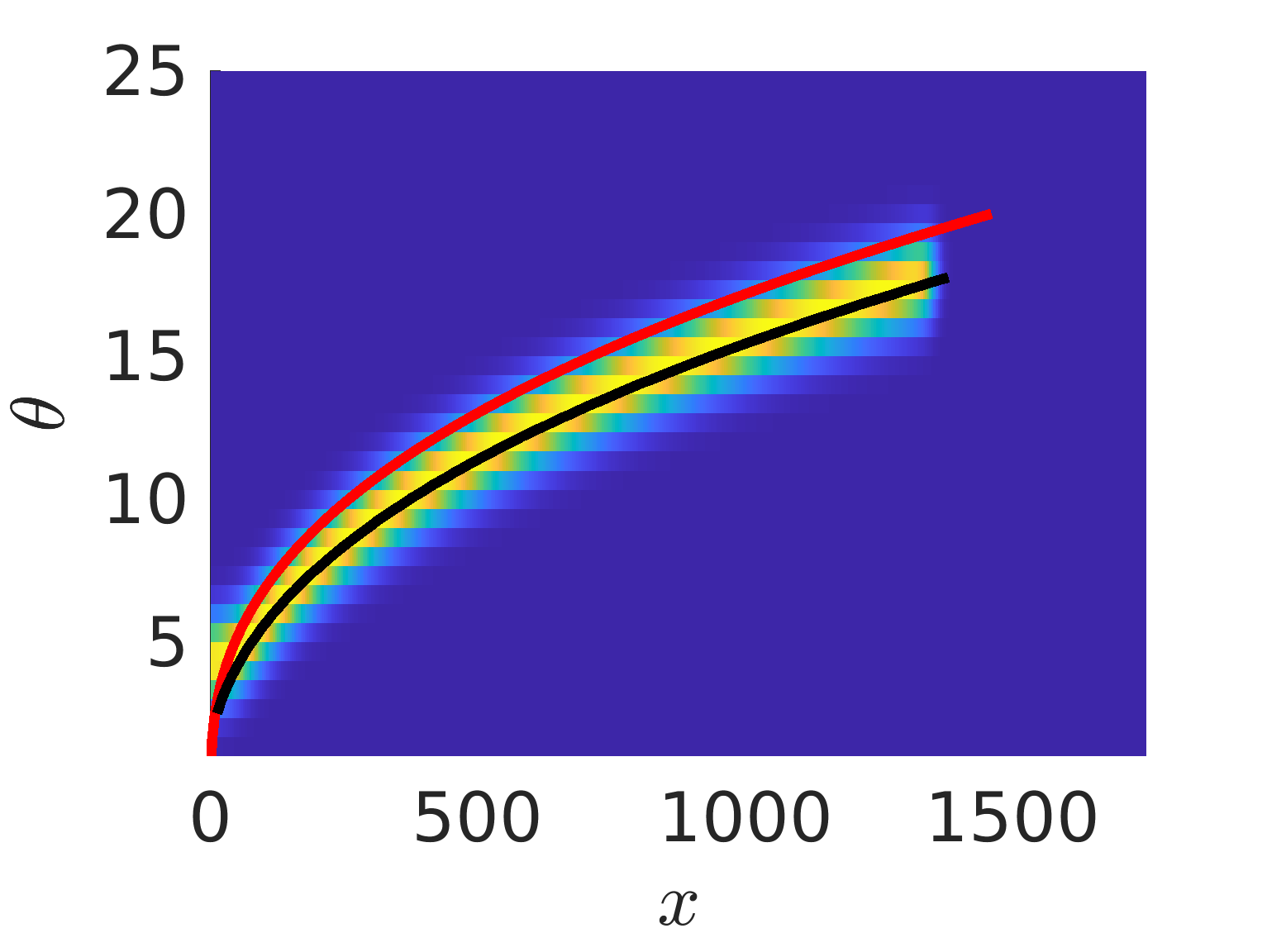}}
  \\
  \subfloat[]{\includegraphics[scale=0.6]{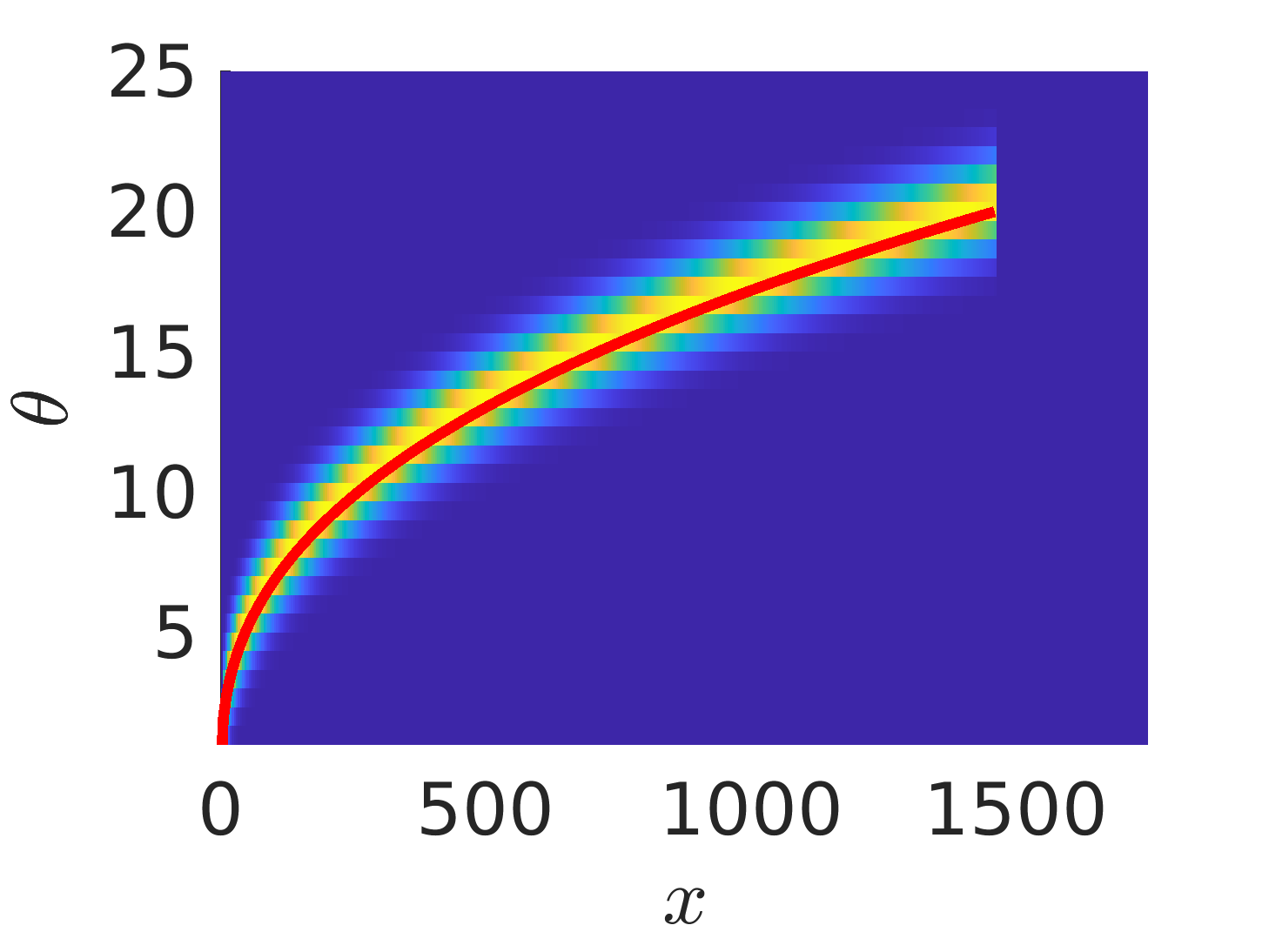}}
  \caption[Contour lines of the trait distribution of a sexual population]{\textbf{Contour lines of the trait distribution of a sexual population}, associated to the Eq. \eqref{chap_cemracs_sex:eq:PDE_apres_simplification_et_variance_generale} with parameters $\delta t = 0.02$, $\delta x = 4$, $\delta \theta =2 /3$, $x_{\max} = 3000$ and $\theta_{\max} = 201$. (a) Trait distribution given by the numerical simulations, at $t=200$. (b) Trait distribution behind the {\color{black}propagating front} given by \cref{chap_cemracs_sex:thm:main}, at $t=200$. The red line represents the approximation of the mean trait behind the {\color{black}propagating front} given by \eqref{meanthetabehind}, and is common to both subfigures, while the dark line is the mean trait behind the {\color{black}propagating front} given by the simulations.
  }
  \label{chap_cemracs_sex:fig2}
\end{figure}

    Using a linear regression on the values for $t\in [60,200]$ (illustrated in Fig.~\ref{chap_cemracs_sex:fig1} (b)), the {\color{black}mean of the dispersive trait} $\bar\theta^{num}$ can be approximated by:
    \[
    \bar\theta^{num}(t) \approx {\color{black}1.02}\, t^{{\color{black}0.54}}, \quad (R^2=1, \;p\text{-value} < 10^{-14}).
    \]
    We can compare this relationship with the {\color{black}mean of the dispersive trait} $\bar\theta(t)$ at the front $X(t)$, given respectively by \eqref{meanthetabehind} and \eqref{chap_cemracs_sex:pos_front}:
    \[
      \bar\theta(t) = \lambda^{4/5}(6X(t)^2)^{1/5} = 2\lambda \sqrt {t} = \sqrt {2t}.
    \]
{\color{black} We notice a non trivial difference between $\bar\theta(t)$ and $\bar\theta^{num}$, mainly in their prefactors ($\sqrt 2$ and ${\color{black}1.02}$), but also in their exponents ($0.5$ and $0.54$) (see also the gap between the red and black lines in Fig.~\ref{chap_cemracs_sex:fig2}). This seems partly due to numerical inaccuracies resulting from having a bounded trait space (thus disregarding the largest traits) and from numerical scheme errors. One can also note that the asymptotic distribution indicated by \cref{chap_cemracs_sex:thm:main} might not yet be reached at time $200$ (upper time bound in our numerical simulations).
}

Let us turn to the description of the trait distribution behind the front. In Fig.~\ref{chap_cemracs_sex:fig2}, we display the contour lines of the trait distribution at time $t=200$: subfigure (a) is the trait distribution given by the simulations, while (b) is the formal trait distribution (behind the front only) given by \cref{chap_cemracs_sex:thm:main}. Our approximation appears to fit the numerical results. More precisely, the red curve, representing the {\color{black}mean of the dispersive trait} at each position behind the front given by \eqref{meanthetabehind}, yields a good approximation of the numerical {\color{black}mean of the dispersive trait}. Moreover, if we represent the numerical trait distribution behind the front at multiple times (see Fig.~\ref{chap_cemracs_sex:fig:film}), we can see that it seems to remain stationary, which is consistent with the fact that the expression of the approximation behind the front given by \cref{chap_cemracs_sex:thm:main} is independent of the time.

\begin{figure}
  \centering
  \subfloat[$t=50$]{\includegraphics[scale=0.4]{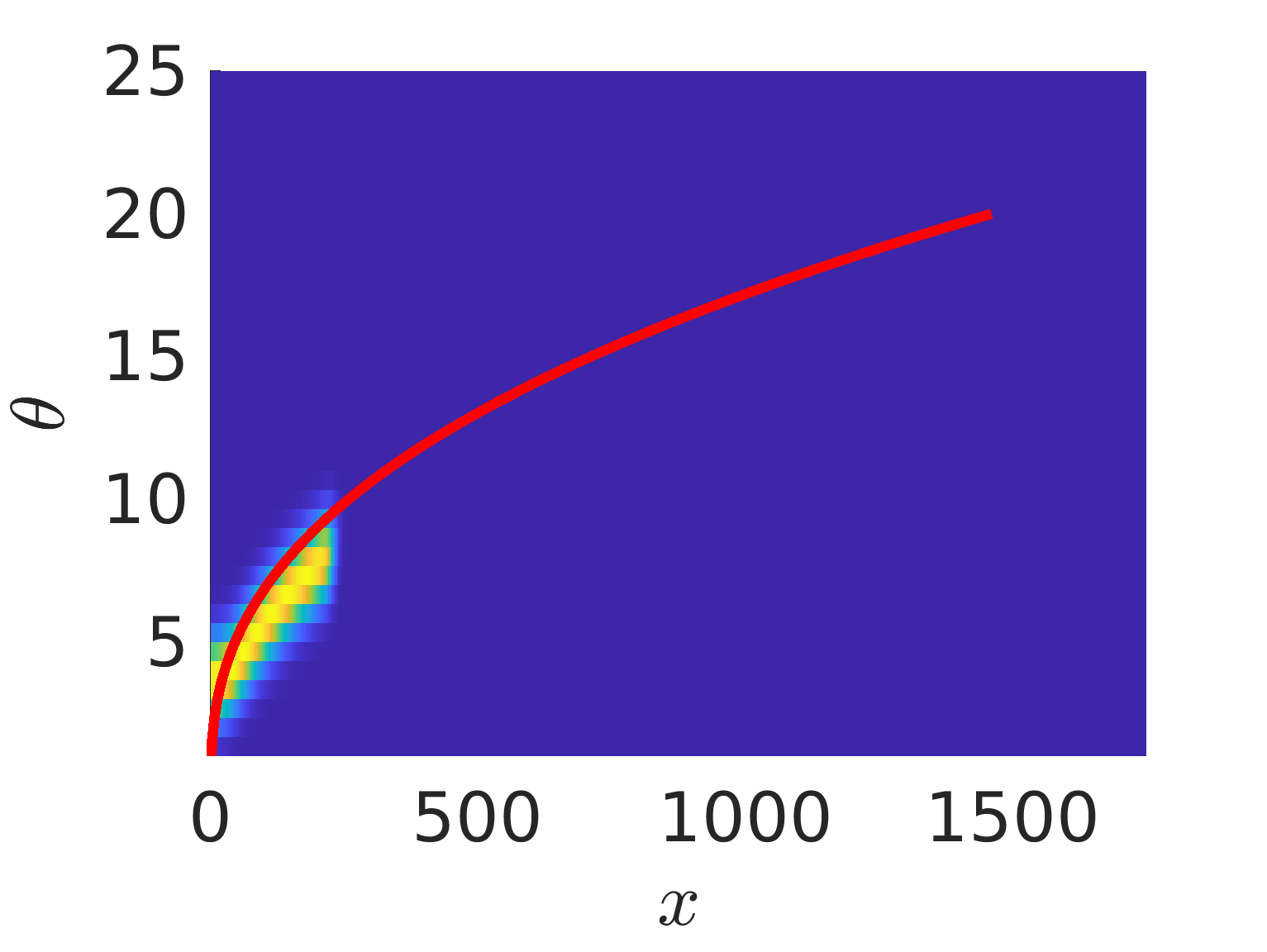}}\hspace{0.3cm}
  \subfloat[$t=100$]{\includegraphics[scale=0.4]{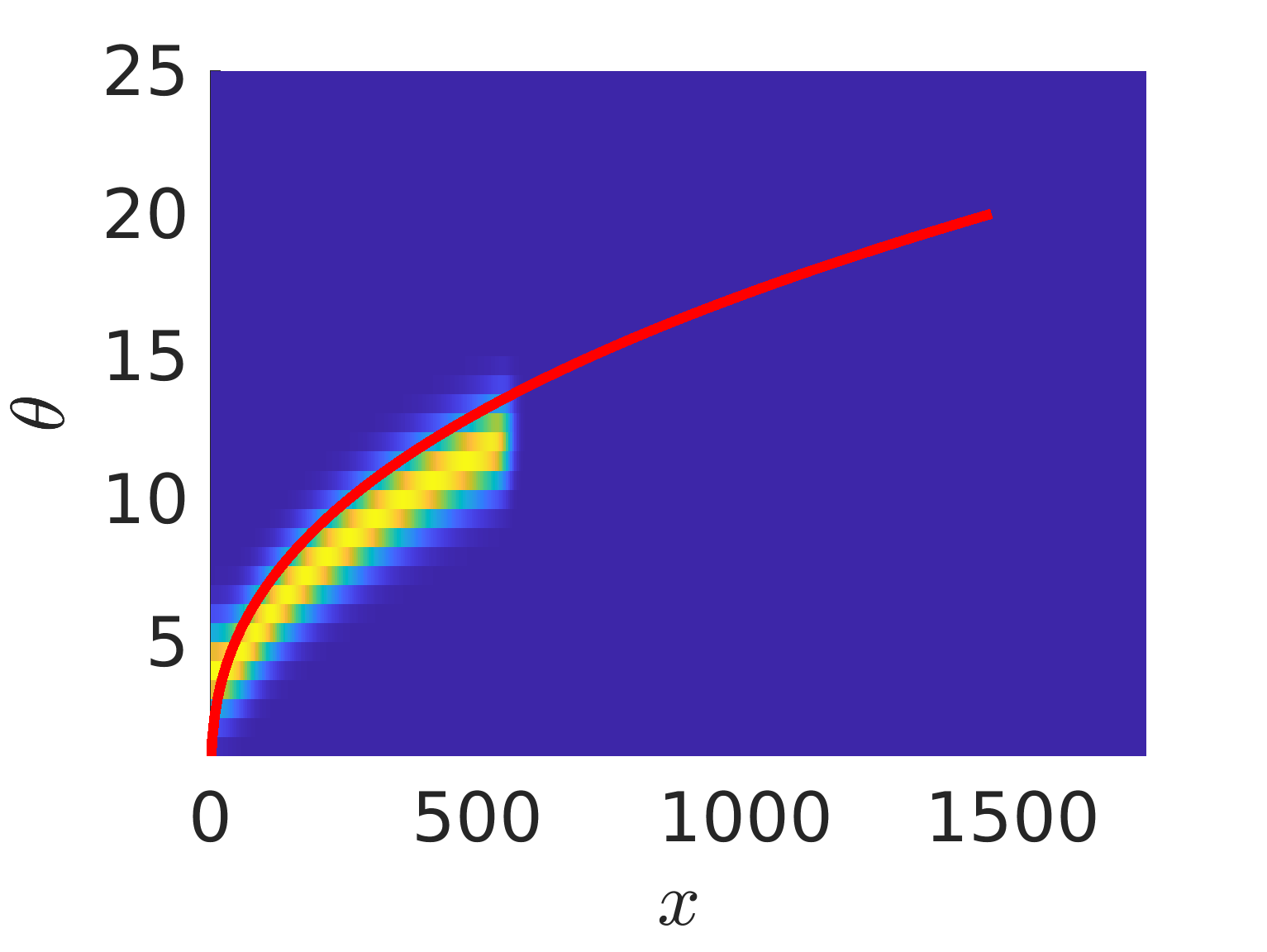}}\\
  \subfloat[$t=150$]{\includegraphics[scale=0.4]{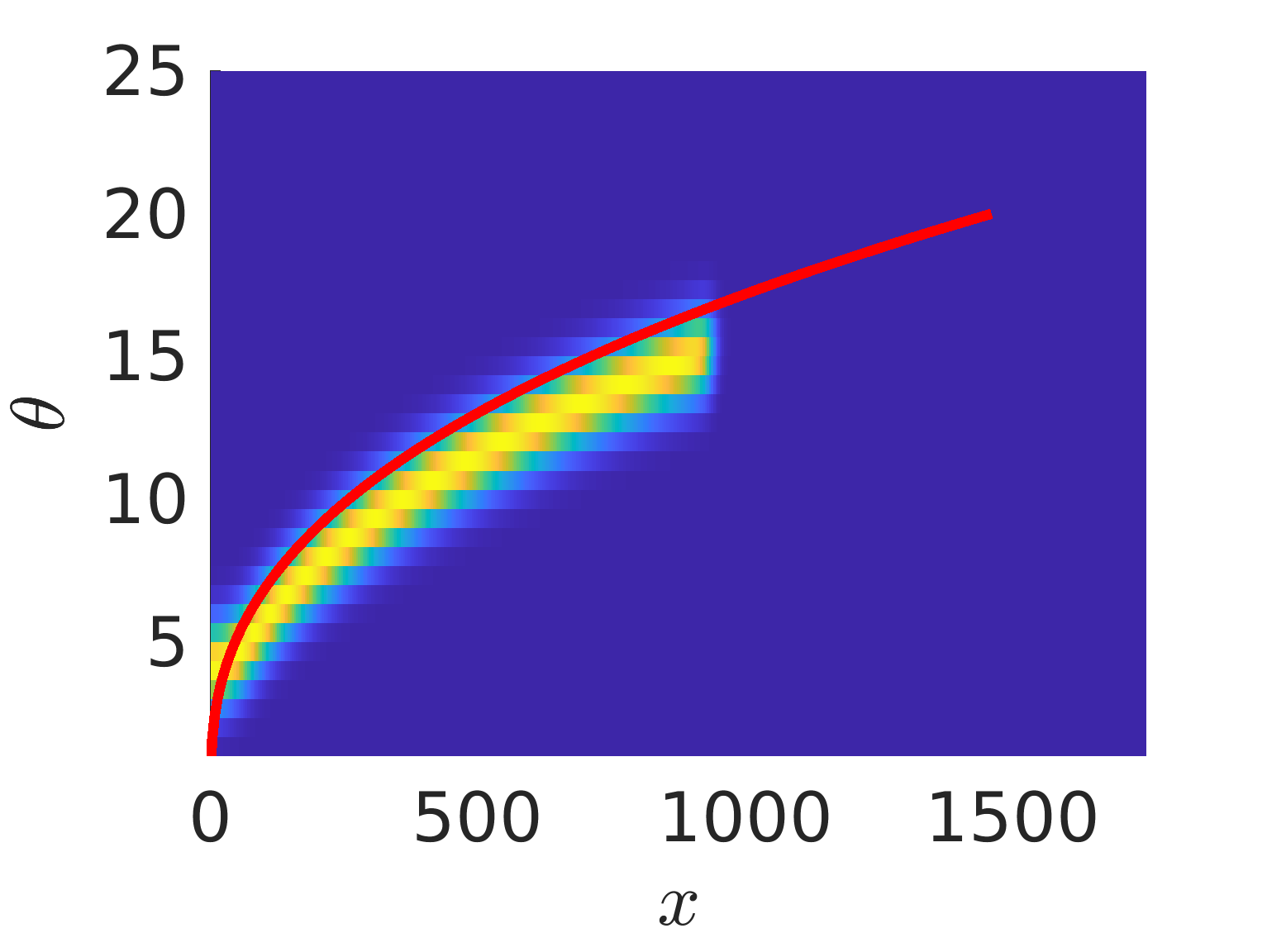}}\hspace{0.3cm}
  \subfloat[$t=200$]{\includegraphics[scale=0.4]{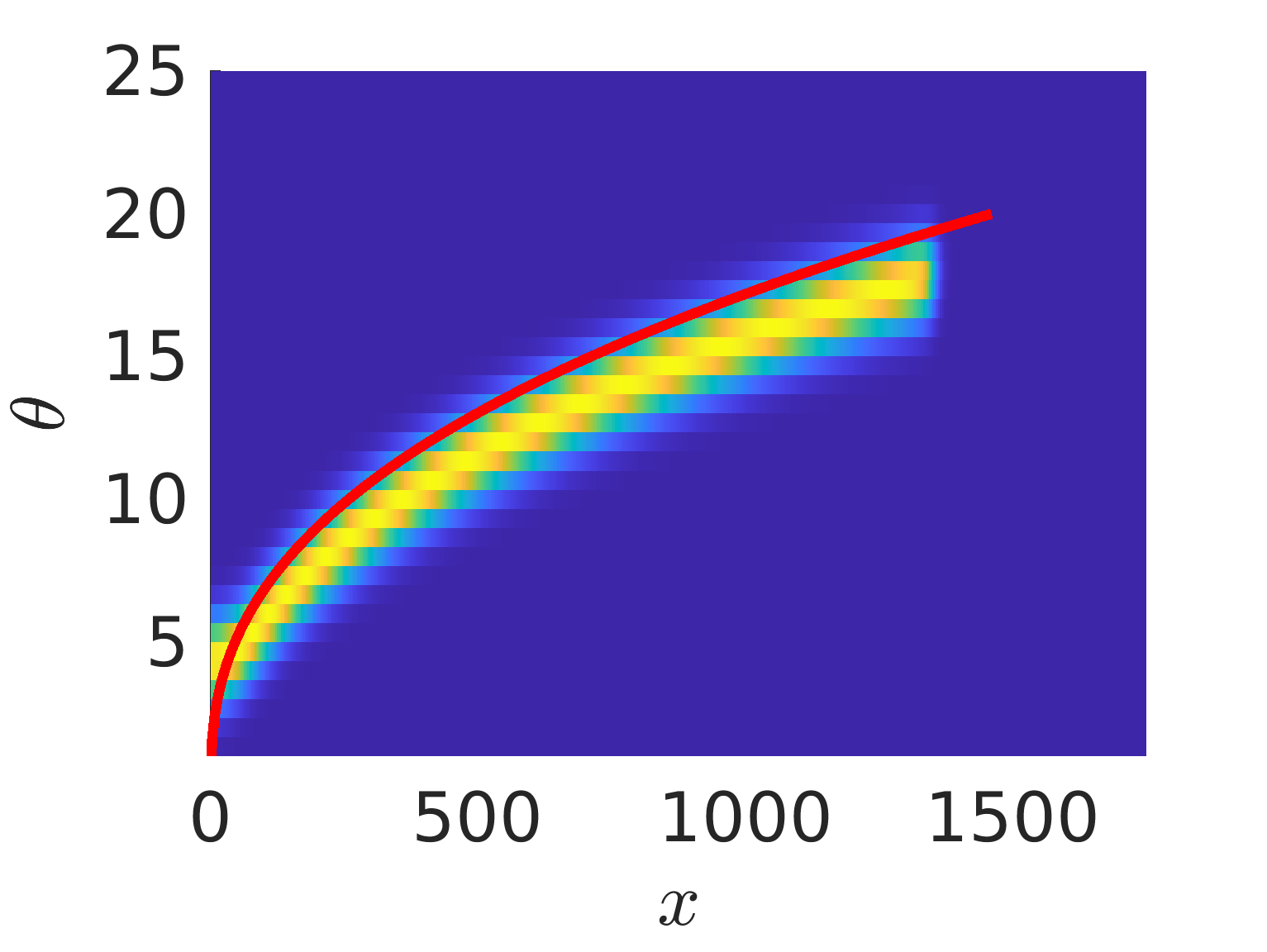}}
  \caption[Contour lines of the trait distribution during the invasion of a sexual population]{\textbf{Contour lines of the trait distribution during the invasion of a sexual population}, given by simulations, at (a) $t=50$ (b) $t=100$ (c) $t=150$ (d) $t=200$. The red line represents the approximation of the mean trait behind the {\color{black}propagating front} given by \eqref{meanthetabehind}, at time $t=200$. The parameters are $\delta t = 0.02$, $\delta x = 4$, $\delta \theta =2 /3$, $x_{\max} = 3000$ and $\theta_{\max} = 201$.}
  \label{chap_cemracs_sex:fig:film}
\end{figure}

\begin{figure}
  \centering
  {\includegraphics[scale=0.3]{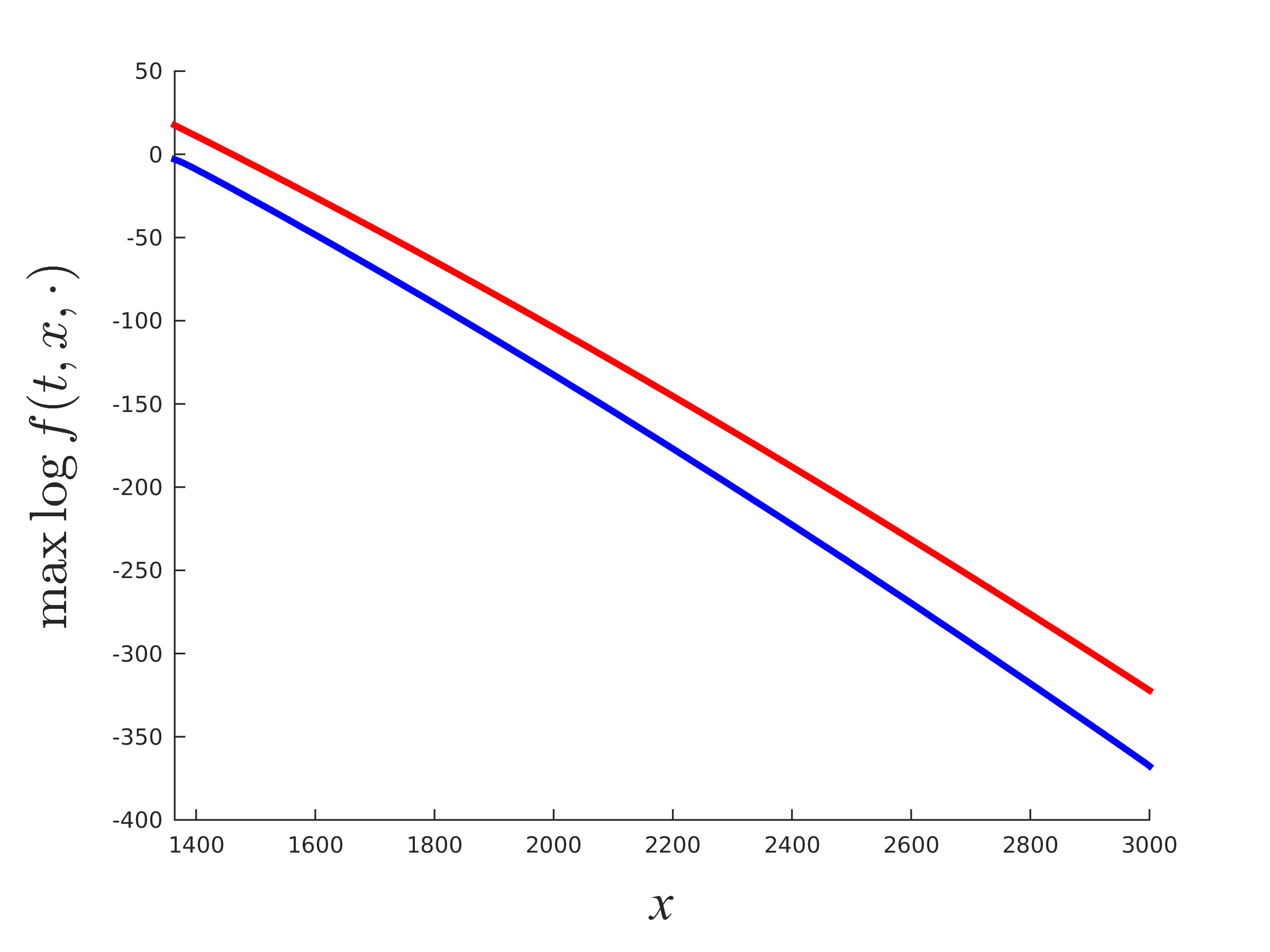}}
  \caption{\textbf{Plot of logarithm of the amplitude of the distribution ahead of the front at time $t=200$.} The blue curve represents the log of the maximum of the distribution $f(t,x,\cdot)$ of the numerical approximation given by the scheme, for different positions $x$ located beyond the numerical value of the front position $(X^{\text{num}}(200) \approx 1400)$. The red curve represents {the prefactor of the trait distribution ahead of the front given by \cref{chap_cemracs_sex:thm:main}}. The parameters are $\delta t = 0.02$, $\delta x = 4$, $\delta \theta =2 /3$, $x_{\max} = 3000$ and $\theta_{\max} = 201$.}
  \label{chap_cemracs_sex:fig:amplitude}
\end{figure}

{
Fig.~\ref{chap_cemracs_sex:fig:amplitude} shows the evolution of the amplitude of the trait distribution $f(t,x,\cdot)$ ahead of the front, in blue curve (log scale). We can see that it can be approximated by the red curve, {which displays the prefactor of the trait distribution ahead of the front given by \cref{chap_cemracs_sex:thm:main}, and that this approximation holds even at very low density}. The difference is due to the other terms of higher power, which are neglected. 
}

\section{Formal proof of the results}
\label{chap_cemracs_sex:sec:proof}

This section is devoted to the formal proof of \cref{chap_cemracs_sex:thm:main}. In \cref{chap_cemracs_sex:proof1}, we set the self-similar variables framework suitable to capture the asymptotic invasion acceleration process. Then in \cref{chap_cemracs_sex:proof2}, we formally derive an asymptotic equation that will allow us in \cref{chap_cemracs_sex:proof3} to determine the position of the front and to derive an approximation of the trait distribution $f(t,x,\theta)$  by finding a solution to the limit problem.

\subsection{Preliminaries}
\label{chap_cemracs_sex:proof1}
According to the same methodology used in previous stu\-dies that model the evolution of dispersion (see for instance \cite{Bou17,proceeding,calvez2018nonlocal}), we define the function $u$ such that:
\begin{equation}\label{chap_cemracs_sex:eq:defu}
f(t, x, \theta) = \exp\left[ - t \, u\left(s(t), t^{ - 5/4}x, t^{ - 1/2}\theta\right)\right],
\end{equation}
where $s(t) = \log(t)$ is a time parametrization (chosen so that $ts^\prime (t)=1$). {\color{black}According to the formal arguments of \cite{proceeding}, we also scale the spatial variable ($y= t^{ - 5/4}x$) and trait variable ($\eta = t^{ - 1/2} \theta$), {\color{black}which leads to the spatial invasion rate accelerating proportionally to $t^{5/4}$} (see \cite{proceeding} for details).} Like in the latter, we recall that the power exponents are chosen so that the all biological forces (particularly, migration and reproduction) contribute in a balanced way in the following PDE on $u$, satisfied for all $t\ge 0$, for all $y\in \R$ and for all $\eta \ge e^{ - s/2}$:
\begin{multline}\label{chap_cemracs_sex:eq:usex}
 - u(s,y,\eta) - \partial _s u(s,y,\eta)+\frac{5}{4} y\partial _y u(s,y,\eta)+\frac{\eta}{2} \partial_\eta u (s,y,\eta)
  \\
  = \eta \left[(\partial_y u(s,y,\eta))^2 - e^{ - s}\Delta_y u(s,y,\eta)\right]\\
  +\left( I[u](s,y,\eta) - \varrho_u(s,y) \right),
\end{multline}
where:
\begin{equation}
  \varrho_u(s,y) = e^{s/2} \int_{e^{ - s/2}}^\infty \exp\left[ - e^{s}u(s, y, \eta)\right]\md\eta,
\end{equation}
and:\footnotesize
\begin{multline}
  I[u](s,y,\eta)= \frac{e^{s}}{\sqrt{2\pi\lambda^2}\varrho_u(s,y)}\\
  \iint_{(e^{ -s/2},\infty)^2} \exp\left[e^{s}\left( - \frac{\left(\eta - \frac{\eta_1+\eta_2}{2}\right)^2}{2\lambda^2} +\left[u(s,y,\eta) - u(s,y,\eta_1) - u(s,y,\eta_2)\right]\right)\right] \, \md\eta_1 \, \md\eta_2.
\end{multline}\normalsize


Henceforth, we note for the sake of clarity: $\alpha=5/4$ and $\beta=1/2$. {\color{black}We generalise also the notation $o_{\epsilon\rightarrow 0}(\epsilon^p)$ for a sequence of functions $r_{\epsilon}(y,\eta)$ by :
\vspace{0.3cm}
\begin{center}
    $r_\epsilon(y,\eta)=o_{\epsilon\rightarrow 0}(\epsilon^p)$ if $\sup_{(y,\eta)}|\epsilon^{-p} r_\epsilon(y,\eta)|$ goes to 0, as $\epsilon$ vanishes.
\end{center}
\vspace{0.3cm}
}

Our formal aim is to determine the large time behaviour of the solution of \eqref{chap_cemracs_sex:eq:usex}, as $s\to \infty$ (which is equivalent to take $t\to \infty$).

\subsection{Formal asymptotic equation}

\label{chap_cemracs_sex:proof2}

In this subsection, we will derive from \eqref{chap_cemracs_sex:eq:usex} an asymptotic equation in the limit $s \rightarrow \infty$ that will explicit the interplay between spatial sorting and trait distribution at the front of the solution.
The main idea is to perform a Taylor expansion of $u$. For that purpose, let us define the variation $\epsilon = e^{ - s/2}$. In the line of \cite{Calvez_Garnier_Patout_2019}, we make the following ansatz:
\begin{equation}
u(s, y, \eta) = u_0(y,\eta)+\epsilon^2 u_1(y,\eta) + o_{\epsilon\rightarrow 0}(\epsilon^2).
\label{chap_cemracs_sex:taylor}
\end{equation}

In the next paragraph, we justify the following separation of trait and space variable in $u_0$, where:
\begin{equation}
  u_0(y,\eta) = b(y) +{(\eta - a(y))^2\over 4\lambda^2}.
  \label{chap_cemracs_sex:quadratic}
\end{equation}
where $a$ and $b$ are continuous and piecewise differentiable functions of the space variable. Let us interpret them.

Using the ansatz \eqref{chap_cemracs_sex:taylor} and \eqref{chap_cemracs_sex:quadratic} in \eqref{chap_cemracs_sex:eq:defu} yields (we recall that $\epsilon = e^{-s/2}$):

\begin{equation}
    f(s,y,\eta) = \exp\left[-\, \frac{b(y)}{\epsilon^2}\right]\exp\left[-\,\frac{(\eta-a(y))^2}{4\lambda^2\epsilon^2}\right]\exp\left[-u_1(y,\eta) +  \underset {\epsilon\rightarrow 0} o (\epsilon^2)\right].
    \label{ansatz}
\end{equation}

Hence, when $s\rightarrow \infty$, the leading term of the trait distribution $\eta\mapsto f(s,y,\cdot)$ is Gaussian, and the correction is brought by a term determined by $u_1$. The space dependent functions $a$ and $b$ crystallize the main effect of spatial sorting on the trait distribution:
\begin{itemize}
    \item[$\diamond$] $a(y)$ gives the mean rescaled dispersal trait $\eta>0$ at position $y$. It is therefore positive and satisfies the relation: \[
    u_0(y,a(y)))=\min\{u_0(y,\eta),\text{ with } \eta\in (0,\infty)\};
    \]
{    \item[$\diamond$] $b(y)$ determines the prefactor of this distribution: formally, we will see that if $b(y)>0$, $\varrho_u(s,\cdot)$ vanishes when $s$ tends to $\infty$. On the contrary, the set $\{b(y)=0\}$ is associated to that area where $\varrho_u$ is asymptotically non-zero. In the context of a spatial invasion, it corresponds to the spatial area that has already been invaded. Hence, we are searching $b$ such that there exists a constant $y_c$ such that $\{b(y)=0\} = \{y\leq y_c\}$. We can interpret $y_c$ as the rescaled position of the front.}
\end{itemize}
Finally, the space dependent functions $a$ and $b$ are linked to the corrector term $u_1$ by an asymptotic equation that we deduce from \eqref{chap_cemracs_sex:eq:usex} (see below for the details). For $y$ where $a$ and $b$ are differentiable: 

{\footnotesize
\begin{multline}
-b(y)-{(\eta-a(y))^2\over 4\lambda^2} +\alpha y\left[b^\prime (y)-a^\prime(y) \, {\eta-a(y)\over 2\lambda^2}\right]+\beta \eta {\eta-a(y)\over 2\lambda^2} -\eta \left[b^\prime (y)-a^\prime(y) \, {\eta-a(y)\over 2\lambda^2}\right]^2
\\
  = \exp\left[ u_1(y, \eta)+u_1(y,a(y))-2u_1\left(y, {\eta+a(y)\over 2}\right)\right] - \mathbf{1}_{\{y\leq y_c\}}.
  \label{chap_cemracs_sex:eq:fusionaetb}
\end{multline}}
In the next section, we find an explicit solution to \eqref{chap_cemracs_sex:eq:fusionaetb}, which encodes the intertwined relationship between spatial sorting and trait distribution.

\textbf{Explanation for the decomposition of $u_0$ \eqref{chap_cemracs_sex:quadratic}.} We will recall the fundamental steps, more extensively detailed formally in \cite{Boui17} and rigorously in \cite{Calvez_Garnier_Patout_2019} (for a model without any spatial structure). From the Taylor expansion of $u$ given in $\eqref{chap_cemracs_sex:taylor}$, we get the following expression for $I[u]$:
\footnotesize
\begin{align*}
I[u](s,\eta,y) = \frac{1}{\epsilon \sqrt{2\pi \lambda^2}}\iint_{(\epsilon,\infty)^2} \frac{\exp\left[\frac{1}{\epsilon^2}A^0_{y,\eta}(\eta_1,\eta_2)\right]\exp\left[A^1_{y,\eta}(\eta_1,\eta_2)\right]\exp\left[\underset{\epsilon\rightarrow 0}o(1)\right] \md\eta_1 \md\eta_2}{\int_\epsilon^{\infty} \exp\left[ - \frac{u_0(y,\eta')}{\epsilon^2} - u_1(y,\eta')\right]\md\eta'},
\end{align*}
\normalsize
where:
\[
\left\{ \begin{array}{l}
  A^0_{y,\eta}(\eta_1,\eta_2) = - {1\over 2\lambda^2} \left[\eta - \frac{\eta_1+\eta_2}{2}\right]^2 +u_0(y,\eta) - u_0(y,\eta_1) - u_0(y,\eta_2), 
  \\
   \\
  A^1_{y,\eta}(\eta_1,\eta_2) = u_1(y,\eta) - u_1(y,\eta_1) - u_1(y,\eta_2).
\end{array} \right.
\]

Then, we have several considerations to make. First, if we assume that $u_0$ reaches its minimum at a non degenerated-point, then the following modified  expression of the denominator:
\[\int_\epsilon^{\infty} \exp\left[-\frac{1}{\epsilon^2}\left[u_0(y,\eta') - \min u_0(y, . )\right]-u_1(y,\eta')\right]\md\eta',\]
will concentrate, as $\epsilon$ goes to 0, around the minimum of $u_0(y,\cdot)$ and have a finite limit. Therefore it is relevant to introduce it both at the numerator and the denominator:
\footnotesize
\[
\frac{1}{\left[\epsilon \sqrt{2\pi \lambda^2}\right]^2}\frac{\iint_{(\epsilon,\infty)^2} \exp\left[\frac{1}{\epsilon^2}\big(A^0_{y,\eta}(\eta_1,\eta_2)+\min u_0(y, . )\big)\right]\exp\left[A^1_{y,\eta}(\eta_1,\eta_2) +\underset{\epsilon\rightarrow 0}o(1)\right] d\eta_1 d\eta_2}{\frac{1}{\epsilon \sqrt{2\pi \lambda^2}}\int_\epsilon^{\infty} \exp\left[-\frac{1}{\epsilon^2}\left[u_0(y,\eta') - \min u_0(y, . )\right]-u_1(y,\eta')\right]d\eta'}.
\]
\normalsize

As we want consequently the numerator not to diverge as $\epsilon \rightarrow 0$, we need that:
\footnotesize
\begin{equation}
\forall \eta \in \R, \underset{(\eta_1,\eta_2)}{\max}\left[- {1\over 2\lambda^2}\left(\eta -\frac{\eta_1+\eta_2}{2}\right)^2+u_0(y,\eta) - u_0(y,\eta_1) - u_0(y,\eta_2) +\min u_0(y, . ) \right] = 0\label{chap_cemracs_sex:min}.
\end{equation}
\normalsize

As shown in \cite{Boui17}, thanks to some convexity arguments, this leads necessarily to choose $u_0(y,\cdot)$ as a quadratic function in $\eta$ with variance $\lambda^2$, hence \eqref{chap_cemracs_sex:quadratic}.

\textbf{Deriving the asymptotic Eq. \eqref{chap_cemracs_sex:eq:fusionaetb} verified by $u_1(\eta,y), a(y)$ and $b(y)$.} To get an asymptotic equation from \eqref{chap_cemracs_sex:eq:usex}, we still need to establish (formally) the limit of $I[u](s,y,\eta)$ as $s=-2\log(\epsilon)$ goes to $\infty$, by incorporating the quadratic expression \eqref{chap_cemracs_sex:quadratic} of $u_0$ in $I[u]$. We will separate the cases of the numerator and the denominator for the sake of clarity.

According to Laplace's method, as we expect the denominator to concentrate around the minimum of $u_0$, namely at $a(y)$, one can perform the change of variable $z:= \frac{\eta'-a(y)}{\epsilon}$:
\footnotesize
\begin{multline*}
\frac{1}{\epsilon \sqrt{2\pi \lambda^2}}\int_\epsilon^{\infty} \exp\left[-\frac{1}{\epsilon^2} \left[u_0(y,\eta') - \min u_0(y, .)\right] - u_1(y,\eta')\right] d\eta' \\
 = \frac{1}{\sqrt{2\pi \lambda^2}}\int_{1-a(y)/\epsilon}^{\infty} \exp\left[-\frac{z^2}{4\lambda^2}\right]\exp\left[-u_1\left[y,a(y)+\epsilon z\right]\right] dz\underset{\epsilon \rightarrow 0}{\rightarrow}\sqrt{2} \exp\left[-u_1\left[y,a(y)\right]\right].
\end{multline*}
\normalsize

Similarly, following the analysis of the authors of \cite{Boui17} and \cite{Calvez_Garnier_Patout_2019} on \eqref{chap_cemracs_sex:min}, we get that the numerator concentrates around the point $(\overline{\eta},\overline{\eta})$, with $\overline \eta = {\eta + a(y)\over 2}>0$, realizing its minimum. One can thus perform the change of variables $(\eta_1,\eta_2)=(\overline{\eta} + \epsilon z_1,\overline{\eta}+\epsilon z_2)$, so that a straightforward computation following the quadratic expression \eqref{chap_cemracs_sex:quadratic} of $u_0$ leads to: 
\footnotesize
\begin{multline}
-\frac{1}{\epsilon^2}\left[-{1\over 2\lambda^2} \left[\eta -\frac{\eta_1+\eta_2}{2}\right]^2+u_0(y,\eta) - u_0(y,\eta_1) - u_0(y,\eta_2) +\min u_0(y, . ) \right] \\
= \frac{1}{4\lambda^2}z_1 z_2 + \frac{3}{8\lambda^2}(z_1^2+z_2^2),
\end{multline}
\normalsize
and therefore:
\footnotesize
\begin{multline*}
 \frac{1}{\left[\epsilon \sqrt{2\pi \lambda^2}\right]^2} \iint_{(\epsilon,\infty)^2} \exp\left[\frac{1}{\epsilon^2}\big(A^0_{y,\eta}(\eta_1,\eta_2) +\min u_0(y, . )\big) \right]\exp\left[A^1_{y,\eta}(\eta_1,\eta_2)+o_{\epsilon\rightarrow 0}(1)\right] d\eta_1 d\eta_2, \\
=\iint_{(1-\overline \eta /\epsilon,\infty)^2} \frac{\exp\left[-\frac{z_1z_2}{4\lambda^2} - \frac{3}{8\lambda^2}(z_1^2+z_2^2)\right]}{[\sqrt{2\pi \lambda^2}]^2}\exp\left[u_1(y,\eta) - u_1(y,\overline{\eta}+\epsilon z_1) - u_1(y,\overline{\eta}+\epsilon z_2)\right] dz_1 dz_2,\\
\underset{\epsilon \rightarrow 0}{\rightarrow} \sqrt{2} \exp\left[u_1(y,\eta) - 2u_1(y,\overline{\eta})\right].
\end{multline*}
\normalsize
We can thereby obtain the formal limit of $I[u]$:
\[
I[u]\left(s,y,\eta\right) \underset{s\rightarrow\infty}{\rightarrow} \exp\left[ u_1(y, \eta)+u_1(y,a(y))-2u_1\left(y, {\eta+a(y)\over 2}\right)\right]. 
\]

Moreover, we need the formal limit of $\varrho_u(s,y)$ as $s=-2\log(\epsilon)$ tends to $\infty$:

\begin{align*}
\varrho _u & \left(-\ 2\log(\epsilon),y\right) = \frac{1}{\epsilon} \int_\epsilon^{\infty} \exp\left[-\frac{u(-2\log(\varepsilon,y,\eta)}{\epsilon^2}\right]\md \eta ,\\
&   = \exp\left[-\frac{b(y)}{\epsilon^2}\right]\frac{1}{\epsilon} \int_\epsilon^{\infty} \exp\left[-\frac{(\eta-a(y))^2}{4\lambda^2\epsilon^2}\right]\exp\left[-u_1(y,\eta)+\underset{\varepsilon\to 0}o(1)\right]\md\eta,\\
&= \exp\left[-\frac{b(y)}{\epsilon^2}\right]\int_{1-\frac{a(y)}{\varepsilon}}^{\infty} \exp\left[-\frac{z^2}{4\lambda^2}\right]\exp\left[-u_1(y,a(y)+\epsilon z)+\underset{\varepsilon\to 0}o(1)\right]\md z.\end{align*}
Hence, formally, we get:
\[\varrho _u\left(-\ 2\log(\epsilon),y\right)\underset{\epsilon \rightarrow 0}{\longrightarrow} \mathbf{1}_{\{b(y)=0\}}2\sqrt{\pi}\lambda \exp\left[-u_1(y,a(y))\right].\]

By integrating all these formal computations in \eqref{chap_cemracs_sex:eq:usex}, we formally obtain an asymptotic equation satisfied by $a$, $b$ and $u_1$, where $a$ and $b$ are differentiable:
{\footnotesize\begin{multline*}
-b(y)-{(\eta-a(y))^2\over 4\lambda^2} +\alpha y\left[b^\prime (y)-a^\prime(y) \, {\eta-a(y)\over 2\lambda^2}\right]+\beta \eta {\eta-a(y)\over 2\lambda^2} -\eta \left[b^\prime (y)-a^\prime(y) \, {\eta-a(y)\over 2\lambda^2}\right]^2\\
  = \exp\left[ u_1(y, \eta)+u_1(y,a(y))-2u_1\left(y, {\eta+a(y)\over 2}\right)\right]\\ -\mathbf{1}_{\{b(y)=0\}}2\sqrt{\pi}\lambda \exp\left[-u_1(y,a(y))\right].
\end{multline*}}

As we are describing a front propagation, we are looking for $a$ and $b$ continuous on $\R$ and differentiable everywhere but not necessarily at the front position (to be determined): \[y_c = \sup\{y,\;b(y) = 0\}.\]
For such functions $a$ and $b$, we have by evaluating the latter at $\eta=a(y)$ for $y<y_c$: \[2\sqrt{\pi}\lambda\exp\left[-u_1(y,a(y))\right] = 1.\] Hence, for $y\neq y_c$  and $\eta \in J_y$ (subset of $\R_+^*$ to be determined), we consider the asymptotic Eq. \eqref{chap_cemracs_sex:eq:fusionaetb}:
{\footnotesize\begin{multline*}
-b(y)-{(\eta-a(y))^2\over 4\lambda^2} +\alpha y\left[b^\prime (y)-a^\prime(y) \, {\eta-a(y)\over 2\lambda^2}\right]+\beta \eta {\eta-a(y)\over 2\lambda^2} -\eta \left[b^\prime (y)-a^\prime(y) \, {\eta-a(y)\over 2\lambda^2}\right]^2\\
  = \exp\left[ u_1(y, \eta)+u_1(y,a(y))-2u_1\left(y, {\eta+a(y)\over 2}\right)\right] -\mathbf{1}_{\{y<y_c\}}.
  \label{pde:y_diff_yc}
\end{multline*}}

\subsection{Resolution of the asymptotic Eq. \eqref{chap_cemracs_sex:eq:fusionaetb}} 

\label{chap_cemracs_sex:proof3}

Let us define for $y\neq y_c$, $\eta>0$:

\begin{multline*}
    g(y,\eta) := -b(y)-{(\eta-a(y))^2\over 4\lambda^2} +\alpha y\left[b^\prime (y)-a^\prime(y) \, {\eta-a(y)\over 2\lambda^2}\right]\\+\beta \eta {\eta-a(y)\over 2\lambda^2} -\eta \left[b^\prime (y)-a^\prime(y) \, {\eta-a(y)\over 2\lambda^2}\right]^2 + \mathbf{1}_{\{y<y_c\}}.
\end{multline*}

Let us fix $y\neq y_c$. For $\eta>0$ such that $g(y,\eta) >0$, we can reformulate \eqref{chap_cemracs_sex:eq:fusionaetb} as:
\begin{equation}
\label{eq:pbu1}T_y(\eta) = L_y(u_1)(\eta),
\end{equation}
where: \begin{equation*}
    T_y(\eta) = \log\left[g(y,\eta) \right],
\end{equation*}
and:
\[L_y(u_1) :\eta \mapsto  u_1(y, \eta)+u_1(y,a(y))-2u_1\left(y, {\eta+a(y)\over 2}\right). \]

Eq.~\eqref{eq:pbu1} suggests that $a$, $b$ and $y_c$ are to be chosen so that $T_y$ lies in the image of the linear operator $L_y$. One can notice that the kernel of $L_y$ is composed of the linear functions, hence: \[\dim \ker\left(L_y\right) = 2.
\]

Heuristically, the image of $L_y$ is orthogonal to a two dimensional space, which is generated by $\delta_{a(y)}$ and $\delta_{a(y)}'$. More precisely, following \cite{Calvez_Garnier_Patout_2019}, one can show that if $T_y$ verifies:
\begin{equation}
  \begin{aligned}
\begin{cases}
T_y\left(a(y)\right)=0,\\
T_y'\left(a(y)\right) = 0,
\end{cases}
\end{aligned}
\label{eq:T_y}
\end{equation}
then the following sum converges:
\begin{equation}
    u_y:\eta \mapsto\sum_{k=0}^\infty 2^k T_y\left[a(y)+\left(\eta-a(y)\right)\,2^{-k}\right],
    \label{eq:u_y}
\end{equation}
and $L_y(u_y) = T_y.$

Hence, we first need to solve \eqref{eq:T_y}, that is to find $y_c>0$, $(a,b)\in C^0(\R)\cap C^1(\R\backslash\{y_c\})$, such that:

\small\begin{equation}
\forall y \neq y_c, \quad\left\{
\begin{array}{l}
   -b(y)+\alpha y b^\prime(y)-a(y) (b^\prime(y))^2+\mathbf{1}_{\{y<y_c\}}=1,
  \\
-\alpha y a^\prime(y)+\beta a(y) - 2\lambda^2(b^\prime (y))^2 + 2 a(y) b^\prime(y) a^\prime (y)=0.
\end{array} \right. 
\label{chap_cemracs_sex:sys:aetbetyc}
\end{equation}

Here, we present an explicit solution to \eqref{chap_cemracs_sex:sys:aetbetyc}: 
\begin{proposition} 
\label{prop:abyc}
Let us define:
\[y_c = 4\sqrt{\frac{\lambda}{3}},\quad \begin{array}{cc}
     &  \\
     & 
\end{array}a:y\mapsto\left\{
\begin{array}{ll}
   {\lambda^{4/5}\,6^{1/5}} \;y^{2/5}, \quad&\text{if}\; y\leq y_c,
  \\
\left(\frac{3\lambda^2}{2}\right)^{1/3}\,y^{2/3},\quad&\text{if}\;y> y_c,
\end{array} \right.\]
and:
\[  b:y\mapsto\left\{
\begin{array}{ll}
   0, \quad&\text{if}\;y\leq y_c,
  \\
\left(\frac{3}{\lambda2^4}\right)^{2/3}y^{4/3} -1,\quad&\text{if}\;y> y_c.
\end{array} \right.\]
Then $a,b \in C^0(\R)\cap C^1(\R\backslash\{y_c\})$ and $y_c$, $a$ and $b$ are solutions of \eqref{chap_cemracs_sex:sys:aetbetyc}.
\end{proposition}

\begin{remark}
The functions $a$, $b$ and $y_c$ given in the previous proposition are the only solutions of \eqref{chap_cemracs_sex:sys:aetbetyc} of the form : $a(y) = Cy^m$, $b(y)=Ky^n-1$ that are positive for $y>y_c$ and continuous in $y_c$.
\end{remark}

To derive a solution for \eqref{chap_cemracs_sex:eq:fusionaetb} from Proposition~\ref{prop:abyc}, one still has to define $T_y(\eta)$, which requires $g(y,\eta)>0$. As $g(y,\cdot)$ is a three order polynomial in $\eta$ with a negative leading coefficient, it is not positive as $\eta$ becomes large so we can not define $T_y$ on $\R_+^*$. However, $a,b$ and $y_c$ are solutions of \eqref{chap_cemracs_sex:sys:aetbetyc}, which is equivalent to:
\[g(y,a(y))=1,\quad \partial_\eta g(y,a(y)) = 0.\] We aim therefore at solving \eqref{chap_cemracs_sex:eq:fusionaetb} locally in $\eta$ around $a(y)$:

\begin{proposition}
Let $a,b$ and $y_c$ be as in Proposition~\ref{prop:abyc}. Then, there exists $J_0$ an interval centered in 1 such that, for all $y\neq y_c$, $\eta>0$ such that $\frac{\eta}{a(y)}\in J_0$, we have $g(y,\eta)>0$. Moreover, for $y\neq y_c$, $T_y = \log(g(y,\cdot))$ is well defined on $a(y)\cdot J_0$ {and for all $J\subset J_0$ open interval centered in 1}:
\begin{itemize}
    \item[$\diamond$] for $y<y_c$ and $\eta \in a(y)\cdot J$, the series defined in \eqref{eq:u_y} converges and is bounded uniformly with regard to $\eta$ and $y$ {and the bound is of the form $A\,|J|^2$}.
    \item[$\diamond$] for $y>y_c$ and $\eta \in a(y)\cdot J$, the series defined in \eqref{eq:u_y} converges and is bounded uniformly with regard to $\eta$, and {the bound is of the form: $B|J|^2y^{8/3}$}.
\end{itemize}
\end{proposition}
\begin{proof}

Since $g(y,\cdot)$ is a polynomial of order three in $\eta$ such that:
\[g(y,a(y)) = 1,\qquad  \partial_\eta g(y,a'(y)) = 0,\] 
we can define $P_y$ polynomial of order three such that:
\[\forall \eta >0,\quad g(y,\eta) = 1-P_y\left(\frac{\eta}{a(y)}\right).\]
As $P_y(1) = P_y'(1)=0$, we get:
\[P_y(X) = (X-1)^2\left[\gamma X+P_y(0)\right],\]
where $\gamma>0$ is the leading coefficient of $P_y$.

We next compute, for $y\neq y_c$ (by continuity for $P_y(0)$):
\[\gamma = \frac{a'(y)^2a(y)^3}{4\lambda^4},\quad P_y(0) = b(y)+\frac{a(y)^2}{4\lambda^2}-\alpha y\left[b^\prime (y)+a^\prime(y) \, \frac{a(y)}{2\lambda^2}\right]+ \mathbf{1}_{\{y>y_c\}}.\]
Hence (adopting the notations $K_{a^-}$,$K_{a^+}$ and $K_b$ such that for $y<y_c,\, a(y) = K_{a^-}y^{2/5}$ and for $y>y_c,\, a(y) =K_{a^+}y^{2/3},\, b(y) = K_b y^{4/3}-1$ -- see the previous proposition):
\begin{itemize}
    \item[$\diamond$] for $y<y_c$, $\gamma = \frac{a(y)^5}{25y^2\lambda^4} = \frac{{K_{a^-}}^5}{25\lambda^4}$ and:
    \[
    P_y(0) = \frac{a^2}{4\lambda^2}-\frac{\alpha a'(y)ya(y)}{2\lambda^2} = \frac{a^2}{4\lambda^2}-\frac{5}{4}\cdot\frac{2 a^2}{10\lambda^2}=0.
    \]
    So, in that case, $P_y = \frac{{K_{a^-}}^5}{25\lambda^4}(X-1)^2X:=P(X)$ does not depend on $y$. As $P(1) = 0$, there exists $\delta\in(0,1)$ such that for all $ y < y_c$ and $\eta \in ]a(y) (1-\delta), a(y)(1+\delta)[$, $P\left(\frac{\eta}{a(y)}\right)<1$, hence $g(y,\eta)>0$.
    
    \item[$\diamond$] for $y>y_c$, $\gamma = \frac{4}{9}\frac{a(y)^5}{4y^2\lambda^4}=\frac{K^5_{a^+}}{9\lambda^4}y^{4/3}$ and:

    \begin{align*}
        P_y(0) & = (b(y)+1)-\frac{5}{3}(1+b(y))+\frac{a(y)^2}{4\lambda^2}-\frac{5a(y)^2}{12\lambda^2}, \\
        & = -y^{4/3}\left[\frac{K^2_{a^+}}{6\lambda^2}+\frac{2K_b}{3}\right]= -\gamma y^{4/3}\left[\frac{3\lambda^2}{2K^3_{a^+}}+\frac{6K_b\lambda^4}{K^5_{a^+}}\right],\\
        & =-\gamma y^{4/3}\left[1+6\times\frac{3^{2/3}\lambda^4 2
   ^{5/3}}{2^{8/3}\lambda^{12/3}3^{5/3}}\right]=-2\gamma y^{4/3}.
    \end{align*}
    Hence: $P_y(X) = \gamma y^{4/3}(X-1)^2(X-2)$, thus: $\forall y>y_c, \forall \eta \in ]0,2a(y)[, g(y,\eta)>1>0$.
\end{itemize}

That proves the first part of the proposition. Let us call $J_0$ a closed interval centered in 1 on which, for all $y\neq y_c$, $g(y,.)$ is positive, and on which $T_y$ is therefore well-defined.

Let us now consider $J\subset J_0$ an open interval centered in 1. For $y\neq y_c$, $\eta \in a(y)\cdot J$, let us define, for $k \in \mathbb{N}$:
\[\eta_k := a(y)+\frac{\eta-a(y)}{2^k}.\]
Next, as $T_y(a(y)) = T'(a(y)) = 0$, we get the following:
\[2^k T_y(\eta_k) = 2^k\int_{a(y)}^{\eta_k}T''_y(t)\frac{\eta_k-t}{2}\md t.\]
With the change of variables $s=2^k\left(t-a(y)\right)$, we get:
\begin{equation}
    2^k T_y(\eta_k) = \int_0^{\eta-a(y)}T''_y\left(a(y)+s2^{-k}\right)\frac{\eta - a(y) -s}{2^k}\md s.
    \label{aux}
\end{equation}
$T''_y$ is continuous on $a(y)\cdot J$, so the latter ensures that $\sum_{k\geq 0}2^kT\left(\eta_k\right)$ converges for all $\eta \in a(y)\cdot J$.

Finally, for $y\neq y_c$, we need to uniformly bound $\sum_{k\geq 0}2^kT\left(\eta_k\right)$ with regard to $\eta \in a(y)\cdot J$.
For $y<y_c$, from the first part of the proof, we have:
\[\forall \eta \in a(y) \cdot J, \quad T_y(\eta) = \log\left(1-P\left(\frac{\eta}{a(y)}\right)\right),\]
with $P(X) = \gamma X (X-1)^2$ and $\gamma$ independent of $y$ and $\eta$. Setting:
\[
\begin{array}{ccccl}
    F&:&J&\rightarrow& \R,  \\
    && x&\mapsto &\log\left(1-P(x)\right),
\end{array}
\]
we dispose of a smooth function, independent from $y$ and $\eta$, such that: 
\[
\forall \eta \in a(y)\cdot J, \quad T_y(\eta) = F\left(\frac{\eta}{a(y)}\right),
\]
and therefore $T''_y(\eta) = {F''\left({\eta}/{a(y)}\right)}/{a(y)^2}$. Following \eqref{aux}, we get (writing $|J|$ as the length of $J$):
\[\forall y<y_c, \eta\in a(y)\cdot J,\quad \sum_{k\geq 0}|2^k T_y(\eta_k)| \leq \sum_{k\geq 0}2^{-(k+1)} \|F''\|_{\infty,J} \frac{(\eta-a(y))^2}{a(y)^2}\leq |J|^2  \|F''\|_{\infty,J_0}.\]

For $y>y_c$, we have from above:\small
\[\forall \eta \in a(y) \cdot J,\quad  T_y(\eta) = \log\left(1-y^{4/3}Q\left(\frac{\eta}{a(y)}\right)\right),\]\normalsize
with $Q(X) = \gamma_Q(X-1)^2(X-2)$ ($\gamma_Q$ a constant independent of $y$ and $\eta$). A straight-forward calculus leads to:\footnotesize
\[T''_y(\eta) = -\frac{y^{4/3}}{a(y)^2}\left[\frac{Q''\left(\frac{\eta}{a(y)}\right)}{1-y^{4/3}Q\left(\frac{\eta}{a(y)}\right)}+y^{4/3}\frac{Q'\left(\frac{\eta}{a(y)}\right)^2}{\left(1-y^{4/3}Q\left(\frac{\eta}{a(y)}\right)\right)^2}\right].\]\normalsize
We recall that, additionally, for $y>y_c$ and $\eta\in a(y)\cdot J$, we have: $1-y^{4/3}Q\left(\frac{\eta}{a(y)}\right)>1$. Hence, from \eqref{aux}, we get: \footnotesize
\begin{align*}
    \forall y>y_c, \forall \eta\in a(y)\cdot J, \quad \sum_{k\geq 0}|2^k T_y(\eta_k)| &\leq y^{4/3}|J|^2\left[\|Q''\|_{\infty,J}+y^{4/3}\|{Q'}^2\|_{\infty,J}\right]\\
    &\leq y^{8/3}|J|^2\left[\frac{\|Q''\|_{\infty,J_0}}{{y_c}^{4/3}}+\|{Q'}^2\|_{\infty,J_0}\right].
\end{align*}\normalsize
\end{proof}

The last proposition allows us to complete our solution for \eqref{chap_cemracs_sex:eq:fusionaetb} for $y\neq y_c$ and $\eta \in a(y) \cdot J$, by defining:
\[u_1 : (y,\eta) \mapsto \sum_{k\geq 0}2^k T_y\left(a(y) + \left(\eta-a(y)\right)\,2^{-k}\right).\]
It also highlights the fact that this solution is local in trait around the mean trait $a(y)$.
Finally, we use it in \cref{chap_cemracs_sex:thm:main} to specify  the magnitude of the error terms in our approximation at large times.

\section{Discussion} 
\paragraph{Contributions} In this paper, we have developed a different framework than the one used for the study of asexual populations (\cite{Ber15,Bou17,calvez2018nonlocal}) by using a mixing operator to analyze the behaviour of the propagation front for sexual population. We have formally found an explicit approximation of the trait distribution during the invasion {by finding a solution to the limit problem at large times}. These formal computations have been numerically compared to the solution of \eqref{chap_cemracs_sex:eq:PDE_apres_simplification_et_variance_generale} and thus confirmed. All the computations have been made after having rescaled the partial dif\-fe\-ren\-tial Eq. \eqref{chap_cemracs_sex:eq:general}. By a variable change, we have that, for all growth rate at low density $\boldsymbol{r}>0$, carrying capacity $\boldsymbol{K}>0$ and segregational variance $\boldsymbol{\lambda}^2>0$, for a population with dispersive traits $\boldsymbol{\theta} \ge \boldsymbol{\thmin} >0$, the density $\boldsymbol{f}$ can be approximated at large time $\boldsymbol{t}>0$ by:\small
\[
\boldsymbol{f}(\boldsymbol{t},\boldsymbol{x},\boldsymbol{\theta}) \approx {\boldsymbol{K}\over\boldsymbol{\thmin}} \ \left\{ \begin{array}{l}
  \exp \left[-\ {1\over 4\boldsymbol{\lambda} ^2} \ \left[\boldsymbol{\theta} - \boldsymbol{\lambda}^{4/5}\left( {6 \boldsymbol{r} \boldsymbol{x}^2}\right)^{1/5}\right]^2 \right],\text{ for } \boldsymbol{x}\leq \boldsymbol{y_c}\ \boldsymbol{t}^{5/4}, \\ \\
  \exp\left[\boldsymbol{r} \boldsymbol{t}-\left({9\boldsymbol{x}^4 \over 256 \, \boldsymbol{\lambda}^2\boldsymbol{t}^2}\right)^{1/3}\right] \ \exp\left[ - {1\over 4\boldsymbol{\lambda}^2}{\left[\boldsymbol{\theta} - \left({3\boldsymbol{\lambda}^2 \boldsymbol{x}^2\over 2\boldsymbol{t}}\right)^{1/3} \right]^2}\right] , \\ \\
  \qquad\qquad\qquad\qquad\qquad\qquad\qquad\qquad\ \text{ for } \boldsymbol{x}\geq \boldsymbol{y_c}\ \boldsymbol{t}^{5/4}.
\end{array}\right.
\]\normalsize
with:
\[
\boldsymbol{y_c} = y_c\, \sqrt{\frac{\boldsymbol{\thmin}}{\boldsymbol{r}}}\, \boldsymbol{r}^{5/4} =4 \left[\lambda\over 3\right]^{1/2}\sqrt{\boldsymbol{\thmin}}\, \boldsymbol{r}^{3/4}= 4\left[ {\boldsymbol{\lambda} \over 3} \right]^{1/2}\, \boldsymbol{r}^{3/4}.
\]
{\paragraph{Difference in acceleration rate between asexual and sexual invasive populations}

Our study shows that the effect of spatial sorting only, through the evolution of dispersion, accelerates the speed at which a sexual population invades. The rate of this acceleration, of $t^{5/4}$, is lower than when considering the influence of the same phenomenon on asexual populations ($t^{3/2}$, see \cite{Ber15,Bou17,calvez2018nonlocal}). Mathematically, the blending inheritance property of the infinitesimal model operator reduces the effect of the spatial sorting by crossing extremely dispersive individuals with less dispersive ones, which does not happen for individuals reproducing clonally.}

\paragraph{Extension: Shape of the front}

However, there are still structural questions to answer on the asymptotic behaviour of the front that we can observe numerically. For instance, the additional Fig.~\ref{chap_cemracs_sex:addfig} allows us to study the deformation of the front propagation, more precisely the shape of the transition front. In Fig.~\ref{chap_cemracs_sex:addfig} (a), the spatial distribution $\varrho$ is displayed with respect to a re-centered scale in:
\begin{equation}
  X_{1/2}(t) = \sup\{x\in\R, \varrho(t,x) \,=\,1/2\}.
  \label{chap_cemracs_sex:def:X_un_demi}
\end{equation}
We can observe a flattening of the front shape, as $t\to+\infty$. More precisely, Fig~\ref{chap_cemracs_sex:addfig} (b), displaying $\varrho$ with respect to the re-scaled variable $\left( x-X_{1/2}(t) \right)\,t^{-1/4}$, shows that the shape of the front seems to flatten at order $t^{1/4}$, as the different curves overlap.
\begin{figure}
  \centering
  \subfloat[]{\includegraphics[scale=0.45]{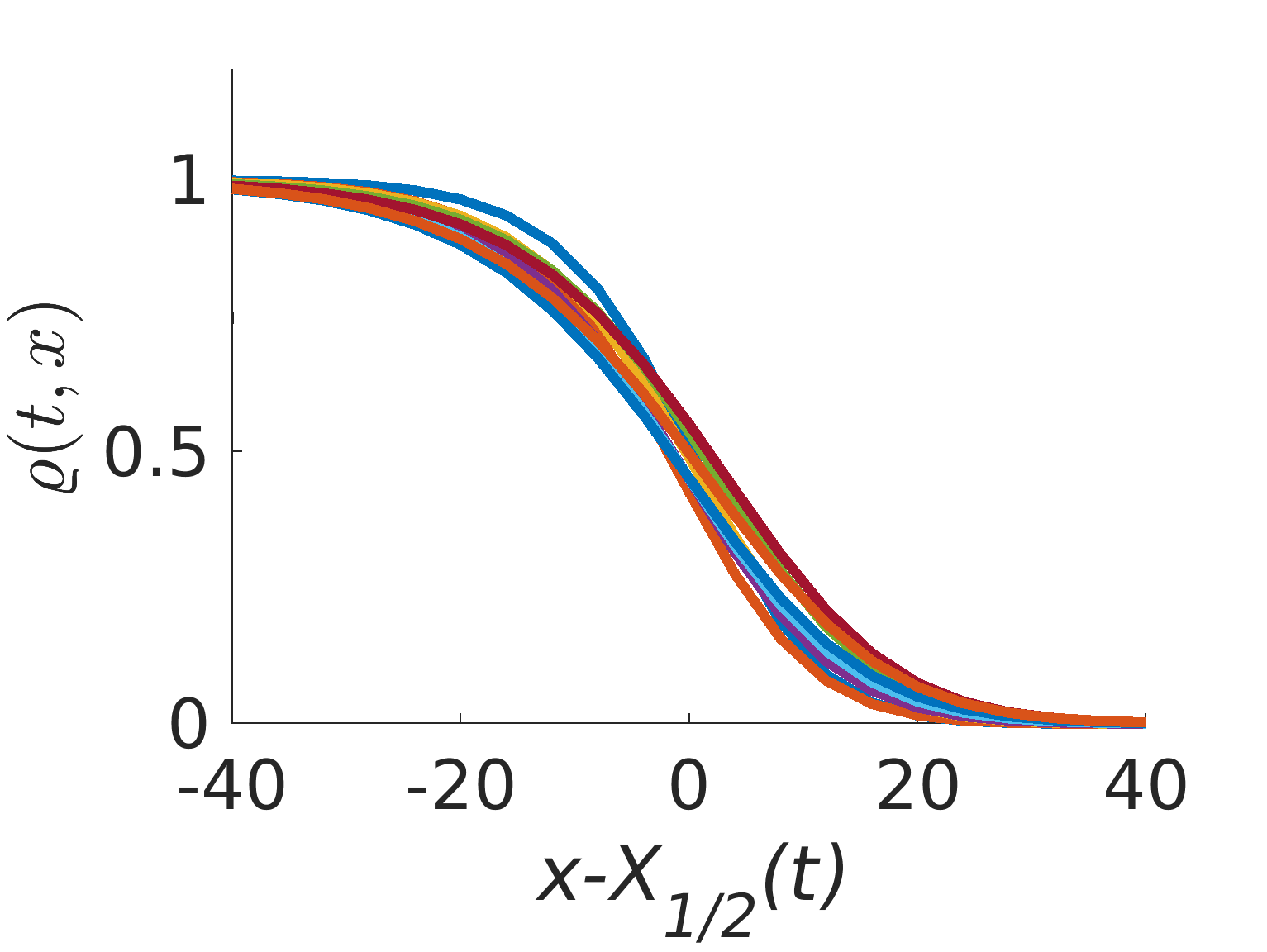}} \hspace{0.3cm}
  \subfloat[]{\includegraphics[scale=0.45]{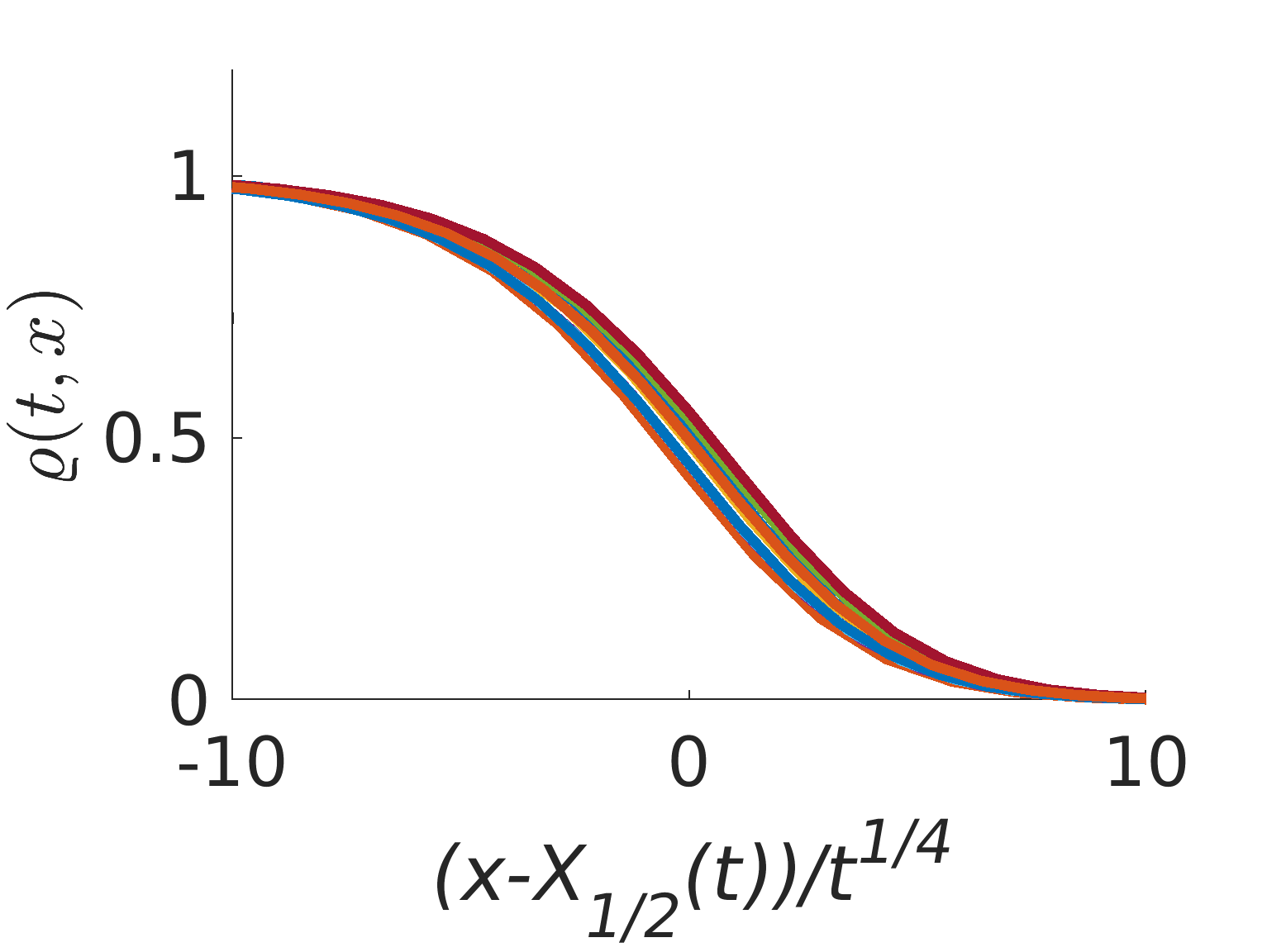}}
  \caption[Plots of the density $\varrho(t,\cdot)$ of a sexual population, with respect to re-centered variable]{\textbf{Plots of the density $\varrho(t,\cdot)$ of a sexual population, with respect to re-centered variables.} The two plots show the evolution of the population density, associated to \eqref{chap_cemracs_sex:eq:PDE_apres_simplification_et_variance_generale}, for successive times at regular intervals from $t=20$ to $t=200$, with respect to (a) the re-centered variable $x-X_{1/2}(t)$, and (b) to the re-scaled variable $(x-X_{1/2}(t))t^{-1/4}$, with $X_{1/2}(t)$ defined in \eqref{chap_cemracs_sex:def:X_un_demi}. The parameters are $\delta t = 0.02$, $\delta x = 4$, $\delta \theta =2 /3$, $x_{\max} = 3000$ and $\theta_{\max} = 201$. {\color{black}Note that the $x$-axis are different between the two plots, for the sake of clarity.}}
  \label{chap_cemracs_sex:addfig}
\end{figure}

\paragraph{Expansion load} 
Here, we consider only a trait linked to the dispersive ability, thus isolating the sole effect of spatial sorting in range expansions, for which there existed no previous precise results. By doing so, our model does not account for any process of selection by adaptation to the local environment. However, in cases of fast range expansion, a phenomenon called the \emph{expansion load} can occur \cite{Peischl_Dupanloup_Kirkpatrick_Excoffier_2013}. As the density of individuals at the front is low, the effective strength of natural selection is reduced allowing deleterious mutations to accumulate at the front. That would eventually undermine the invasion process by reducing the fitness of leading individuals (see \cite{doi:10.1111/j.1461-0248.2010.01505.x}), with the potential effect of slowing down the speed of the front in comparison to the asymptotic formal result of our study. Nevertheless, the clear relationship between the effect of spatial sorting and expansion load is yet to be explored, as a recent analysis using a discrete space framework seems to indicate that the evolution of dispersal rate can prevent expansion load in certain cases (see \cite{Peischl_Gilbert_2018}). By isolating the effect of spatial sorting, our study can therefore constitute a first step in understanding the intricate relationship between the evolution of dispersion and of life history traits, ultimately providing tools to analyse the source of variability in range expansions (see \cite{Williams_Hufbauer_Miller_2019}).

{\color{black}Because the formal computations ignore competition ahead of the front}, even though the simulations seems to va\-li\-da\-te our results, this paper has to be seen as a premise for a consistent and rigorous proof for this problem.  

\section*{Aknowledgement}

This project has received funding from the European Research Council (ERC) under the European Union’s Horizon 2020 research and innovation program (grant agreement No 639638) and from the ANR projects NONLOCAL (ANR-14-CE25-0013) and RESISTE (ANR-18-CE45-0019).

Special thanks to Vincent Calvez and Ga\"el Raoul for the supervision of this CEMRACS project, and to Joachim Crevat for his central role in the genesis of this work. Thanks also to Vincent Calvez, Sepideh Mirrahimi, Lionel Roques, Joachim Crevat, Barbara Neto-Bradley, Linn\'ea Sandell, Gil Henriques, Sarah Otto and Ailene Macpherson for helpful comments.

\printbibliography

\end{document}


\maketitle

\section{Simulations with a different initial distribution} In this section, we explore the robustness of the numerical confirmation of our approximation drawn in the associated article, with regard to the initial conditions. We hereby present the simulations given by the explicit Euler scheme, with the same model parameters as for the figures in the associated article. However, we assume that the initial distribution is a Dirac distribution at $x=0$ and $\theta=1$.

The same conclusions seem to hold: at large time $t$, this invasion seems to follow the same evolution as when being initially normally distributed (\textit{cf.} the associated article). More precisely, we can see that the propagating front accelerates (see Fig~\ref{chap_cemracs_sex:cas_dirac:fig1} (a)) and that the acceleration in space can be quantified similarly: $X(t) =y_ct^{5/4}$, with $y_c >0$.

Fig.~\ref{chap_cemracs_sex:cas_dirac:film} displays the numerical solution {\color{Black}of} our differential equation with a Dirac initial distribution for successive times. It is compared to the {\color{Black} mean of the dispersive trait} (red line) derived from the asymptotic approximation stated in Conjecture~1. Behind the front, the distribution seems to be stationary at large time.

\begin{figure}[!h]
    \centering
    \subfloat[]{\includegraphics[scale=0.65]{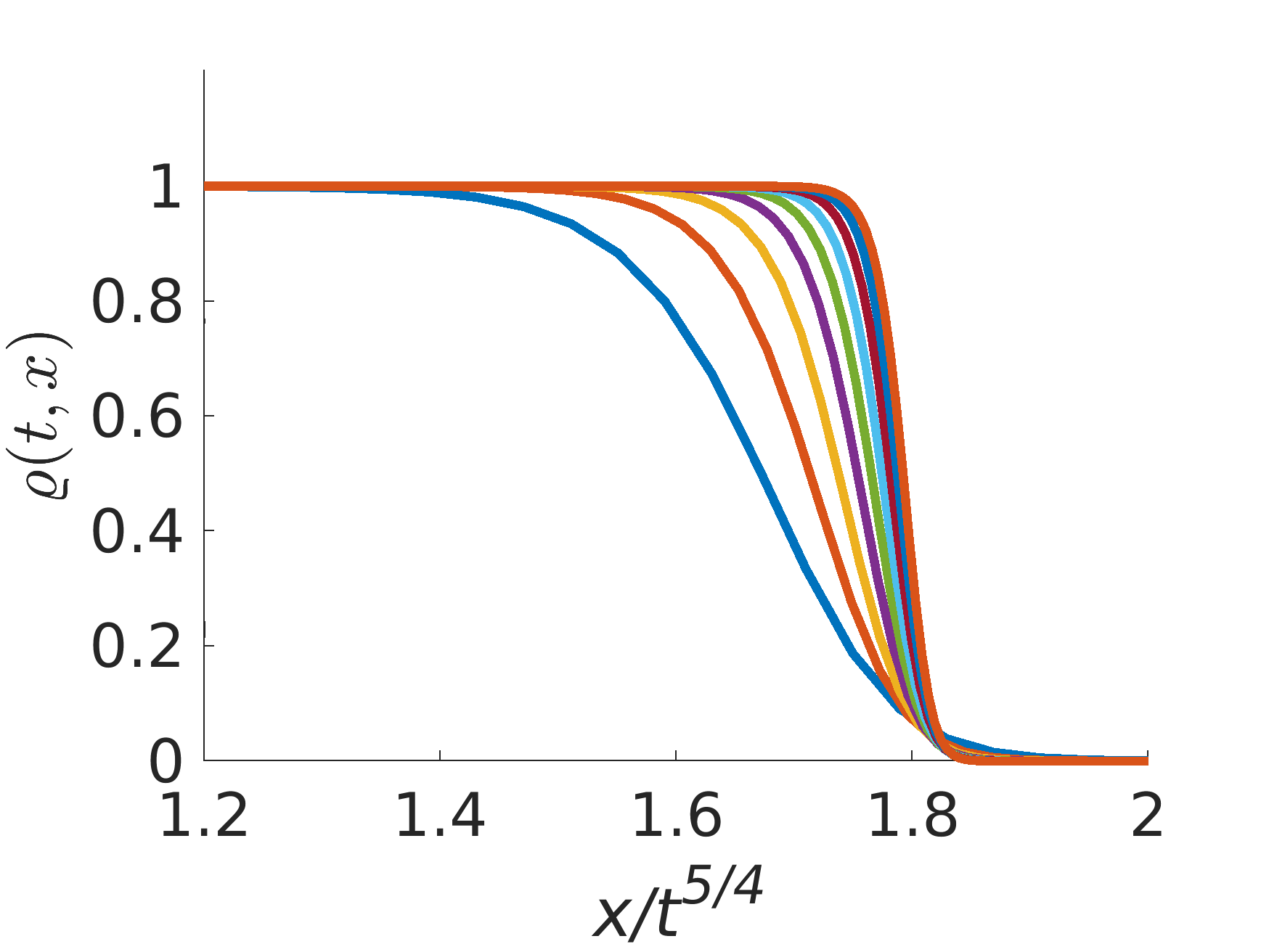}} 
\\    \subfloat[]{\includegraphics[scale=0.65]{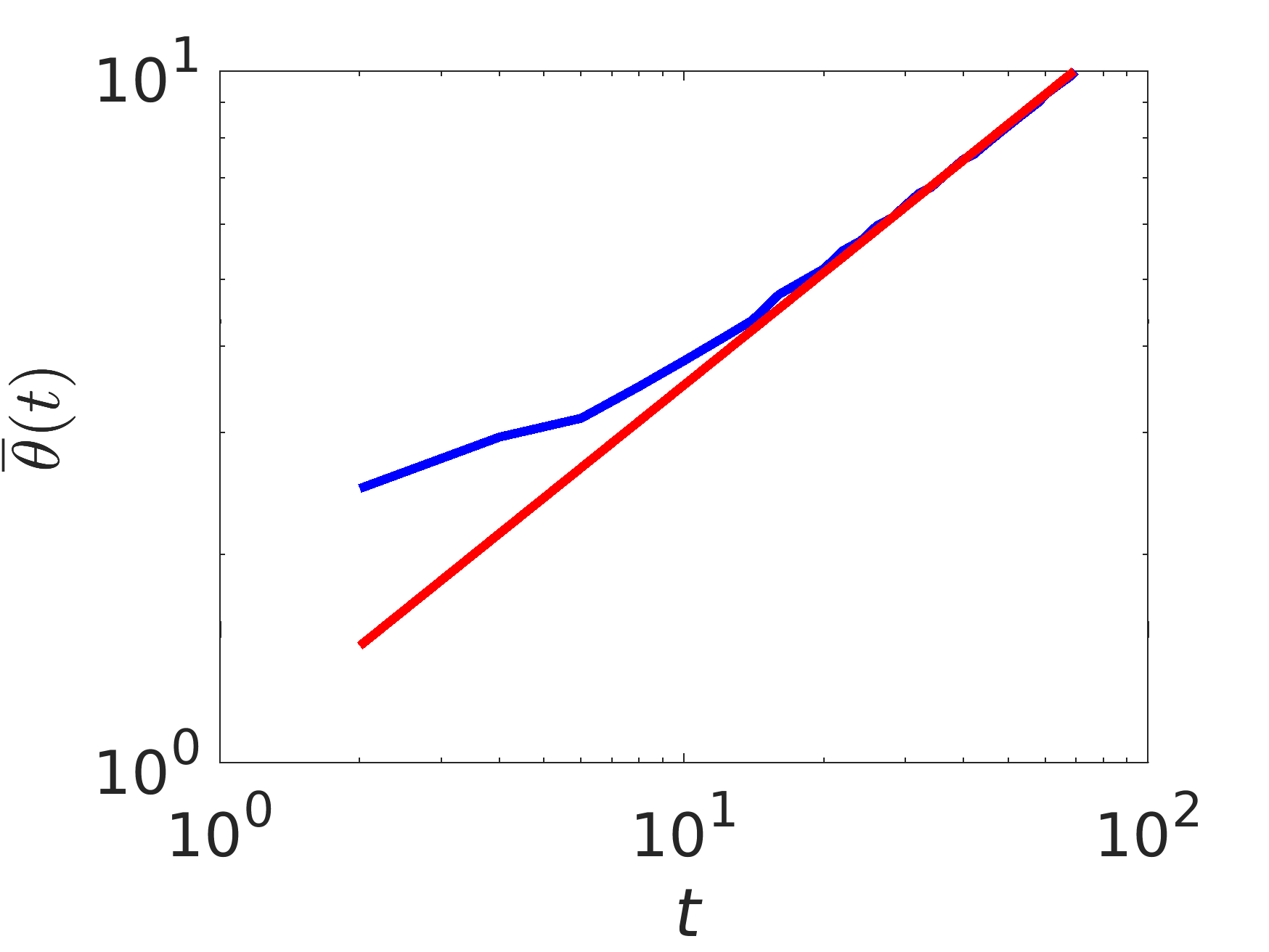}} 
    \caption[Simulations of the invasion of a sexual population, initially distributed according to a Dirac distribution]{\textbf{ Simulations of the invasion of a sexual population, initially distributed according to a Dirac distribution}, with parameters $\delta t =  0.02$, $\delta x = 4$, $\delta \theta =2 /3$, $x_{\max} = 3000$ and $\theta_{\max} = 201$. (a) Plot of the {\color{Black}population size} $\varrho(t,\cdot)$ for successive fixed times at regular intervals from $t=20$ to $t=200$, with respect to  the auto - similar variable $x t^{ - 5/4}$. (b) Plot of the {\color{Black} mean of the dispersive trait} $\overline{\theta}^{num}(t)$  at the front position with respect to time (blue curve) and of the function $t \to {\color{Black}1.02}\, t ^{\color{Black}0.54}$, given by a linear regression with $R^2=1$ (red curve), in $\log - \log$ scale.}
    \label{chap_cemracs_sex:cas_dirac:fig1}
\end{figure}

\begin{figure}[!h]
    \centering
    \subfloat[$t=50$]{\includegraphics[scale=0.4]{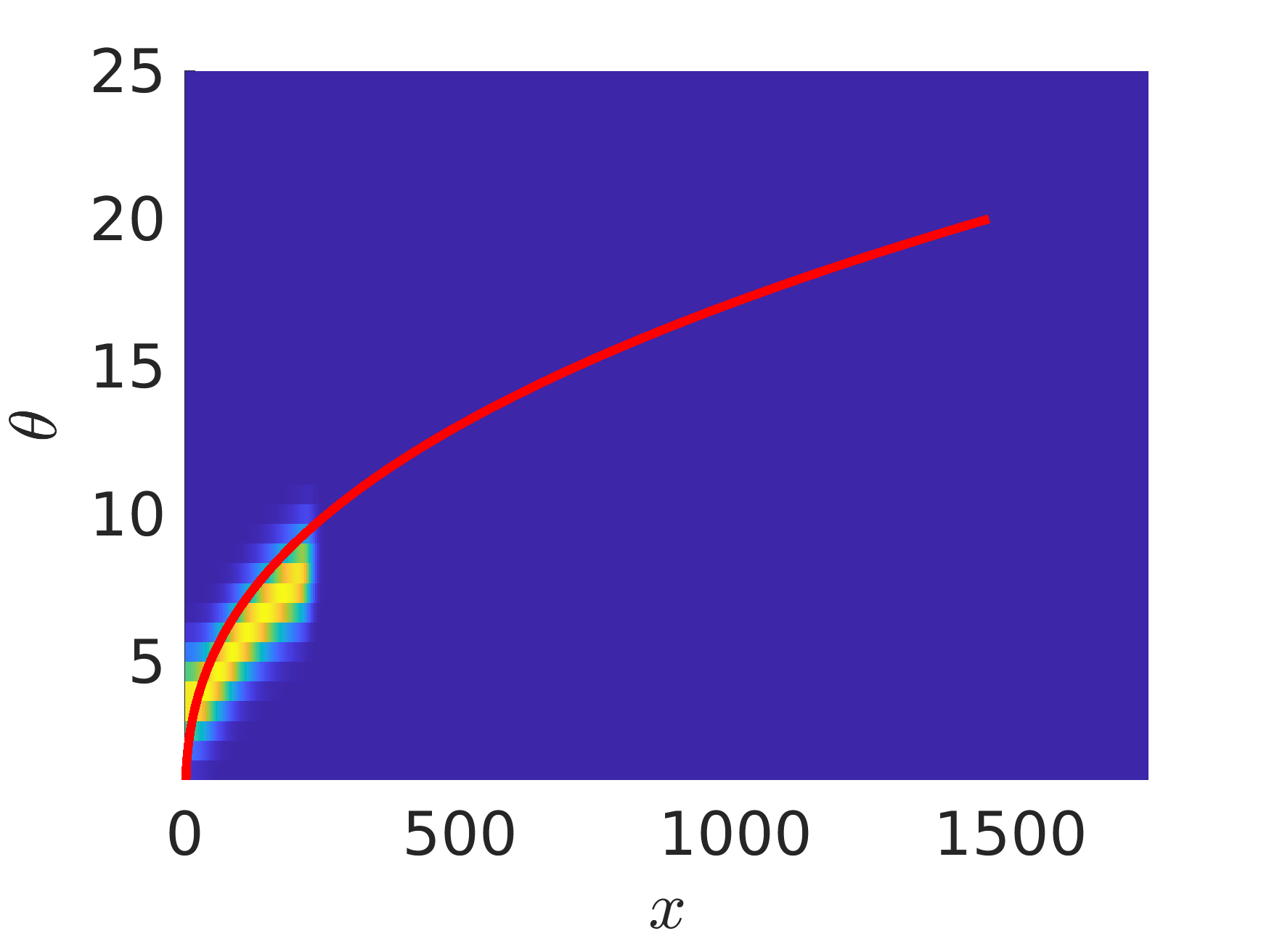}}\hspace{0.3cm}
    \subfloat[$t=100$]{\includegraphics[scale=0.4]{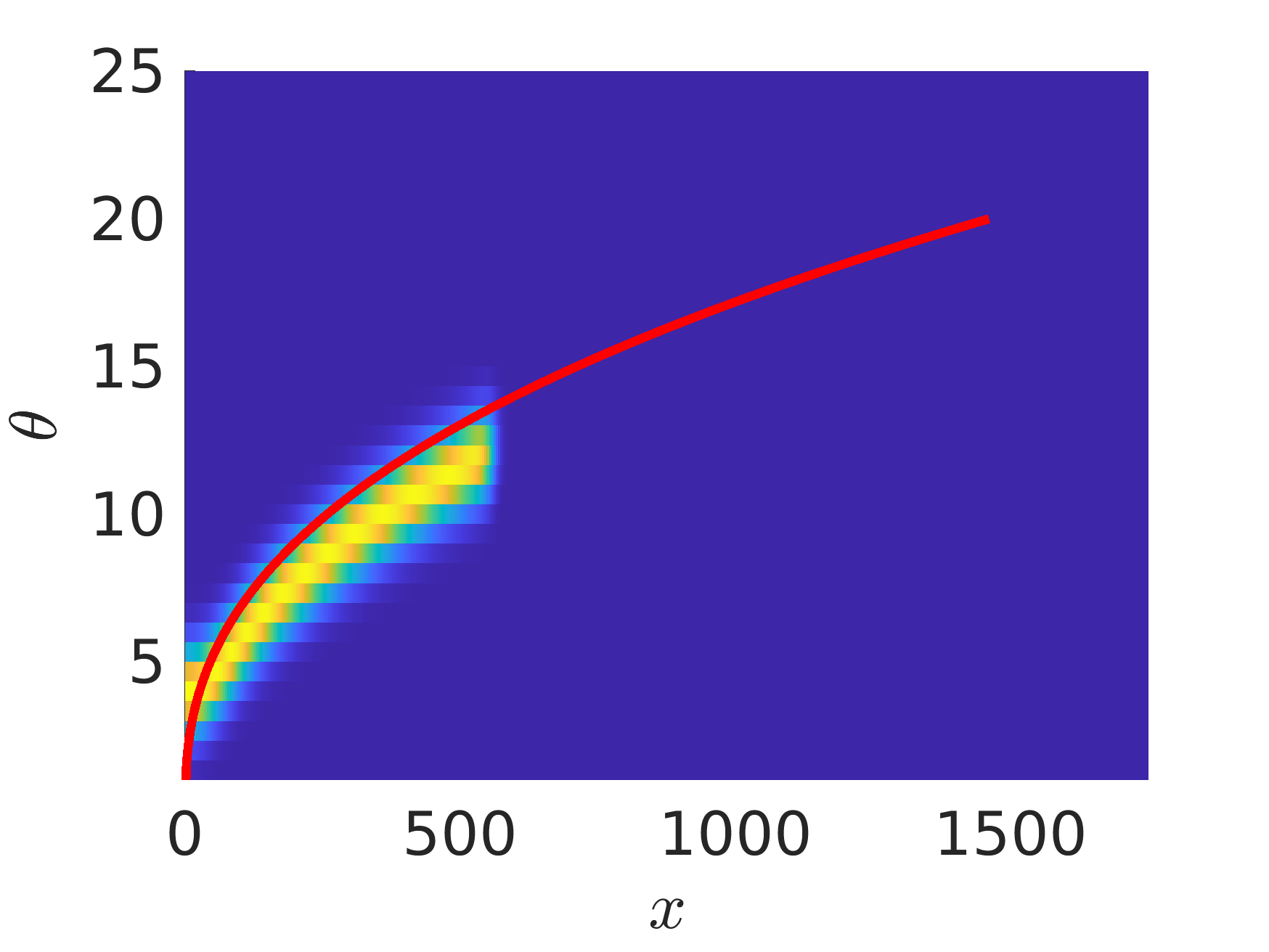}}\\
    \includegraphics[scale=0.4]{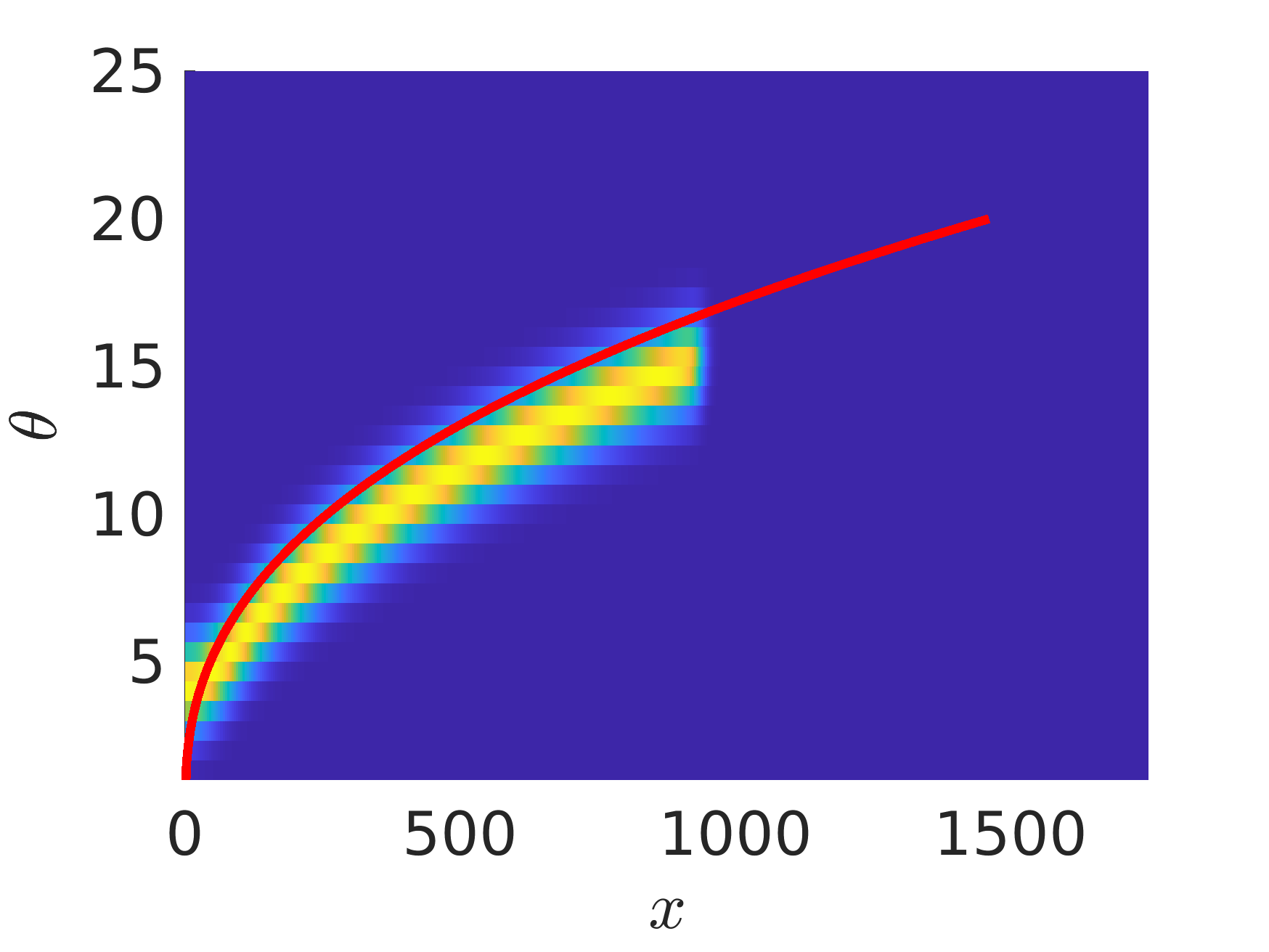}\hspace{0.3cm}
    \subfloat[$t=200$]{\includegraphics[scale=0.4]{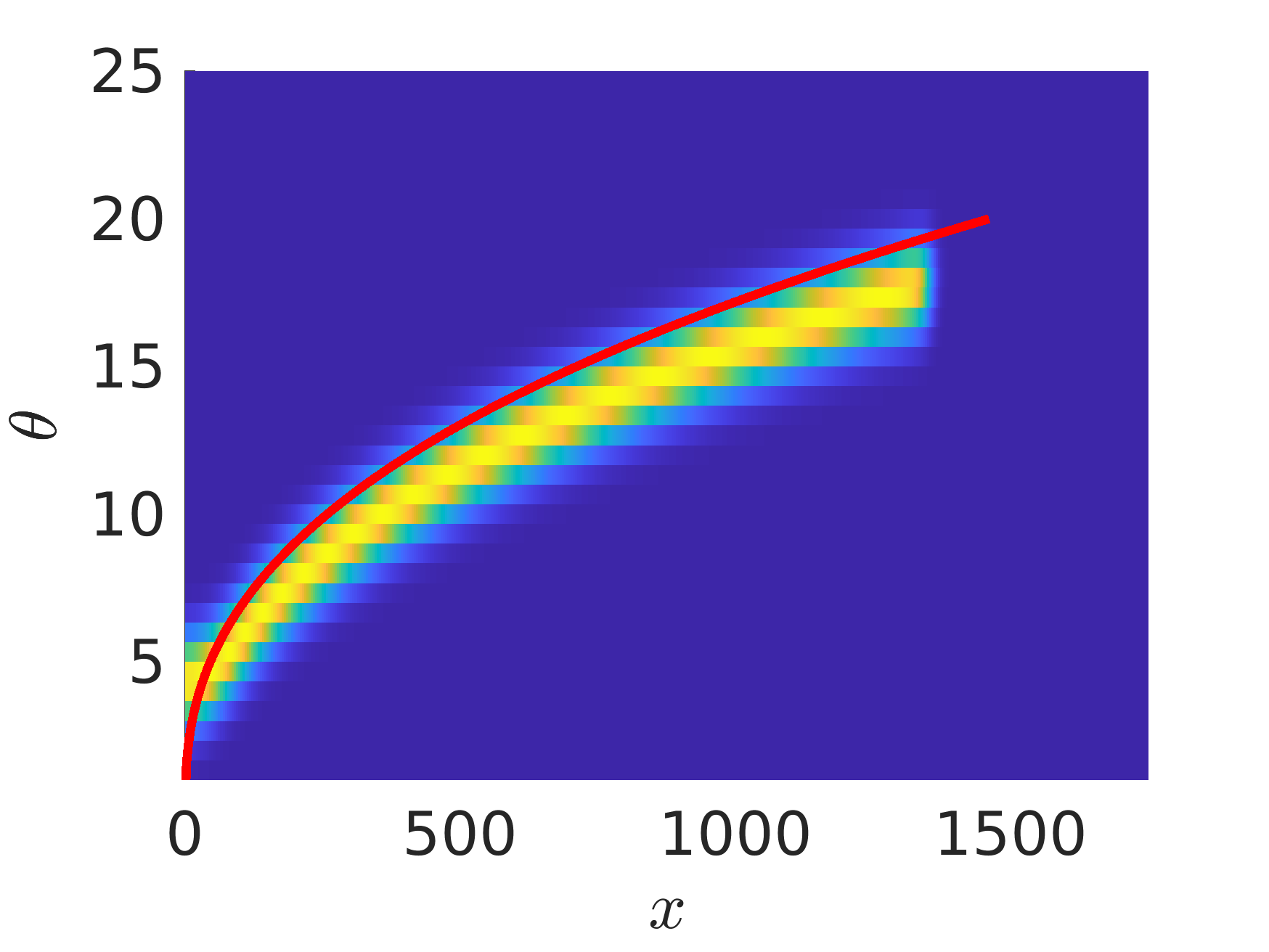}}
    \caption[Contour lines of the trait distribution during the invasion of a sexual population, initially distributed according to a Dirac distribution]{\textbf{Contour lines of the trait distribution during the invasion of a sexual population, initially distributed according to a Dirac distribution}, given by simulations, at (a) $t=50$ (b) $t=100$ (c) $t=150$ (d) $t=200$. The parameters are $\delta t =  0.02$, $\delta x = 4$, $\delta \theta =2 /3$, $x_{\max} = 3000$ and $\theta_{\max} = 201$.  The red line represents the approximation of the mean  trait behind the {\color{Black}propagating front}.}
    \label{chap_cemracs_sex:cas_dirac:film}
\end{figure}

\section{Simulations with a different value for \texorpdfstring{$r$}{Lg}} This supplementary material shows the simulations of the invasion of a sexual population, initially gaussian distributed. We change the value of the growth rate at low density: we take $r=0.1$.

As for the other cases, the {\color{Black}propagating front} accelerates (Fig.~\ref{chap_cemracs_sex:diff_r:fig1}). The front position is given by:
\[
X^{num}(t) = y_c(r=0.1) t^{5/4},\qquad \hbox{for }y_c^{num}(r=0.1)>0.
\]
As seen in the article, the constant $y_c^{num}\approx 0.3$ depends on the growth rate at low density ($r$). We have already discussed about the general case (with a general coefficient $r$). We have guessed that the front at time $t>0$ is at the position:
\[
X(t) = 4\left( \frac{\boldsymbol{\lambda}}  3\right)^{1/2} \, \boldsymbol{r}^{3/4} \ t^{5/4}= {2\over \sqrt 3\  5^{3/4}}\,t^{5/4}\approx 0.35\, t^{5/4},
\]
which is {\color{Black}less consistent} with our simulations {\color{Black}$($maybe because the time is not large enough$)$}. Thus the growth rate at low density $r$ is stronger, the invasion is faster: when individuals have more children, the population can invade faster areas. We present also the evolution of the invasion in Fig.~\ref{chap_cemracs_sex:cas_dirac:film}.

\begin{figure}[!h]
    \centering
    \subfloat[]{\includegraphics[scale=0.65]{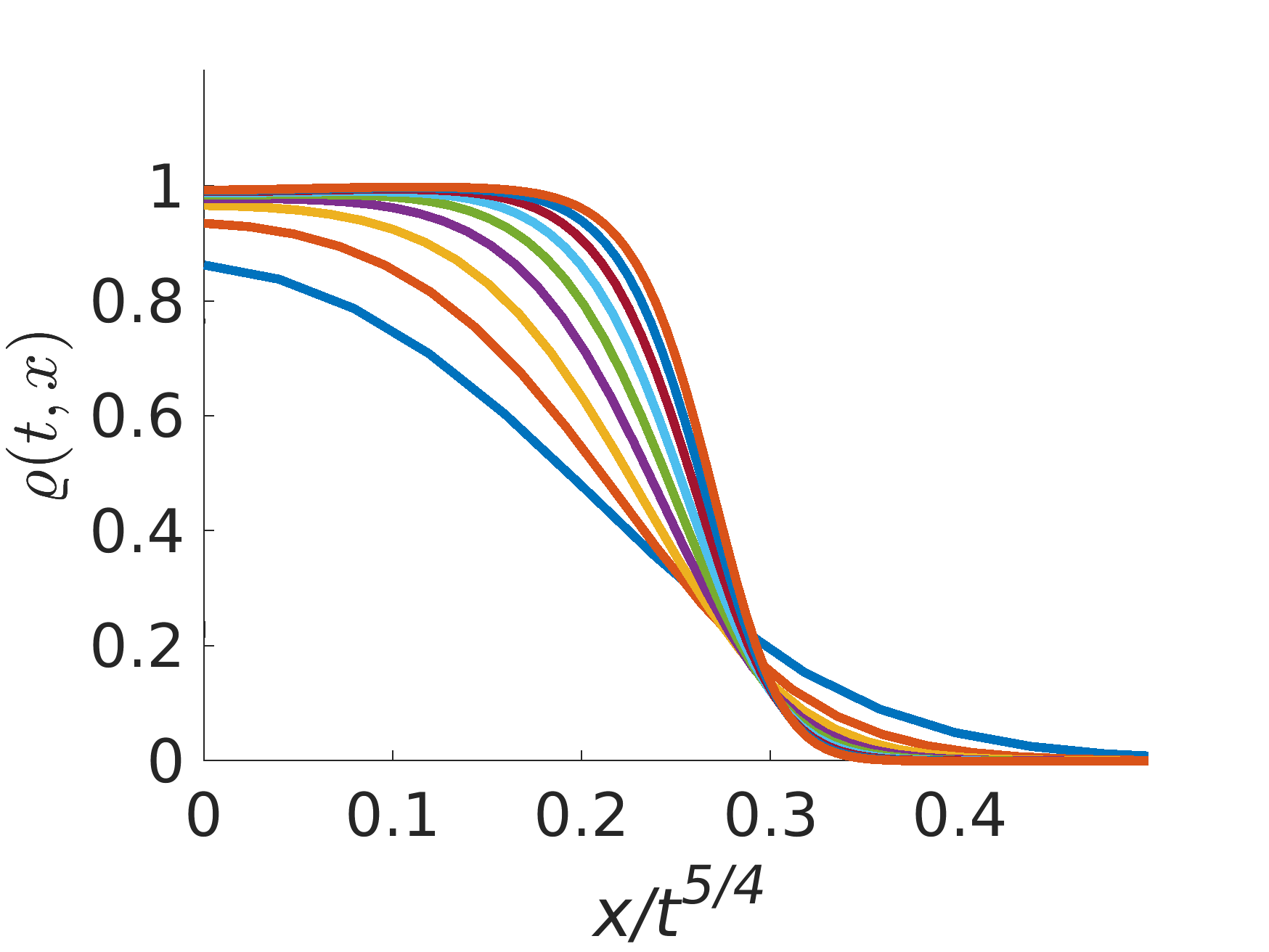}} 
\\    \subfloat[]{\includegraphics[scale=0.65]{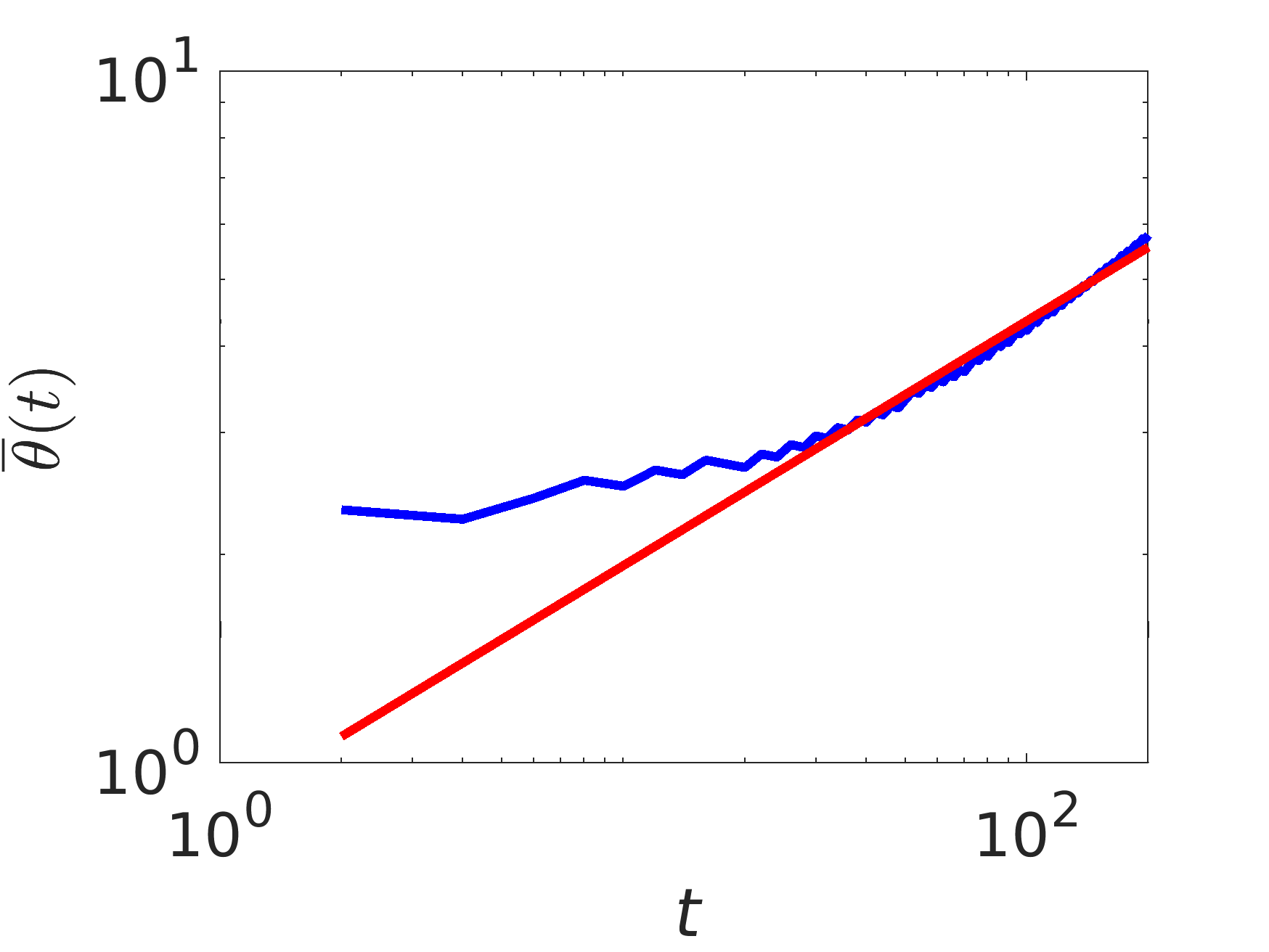}} 
    \caption[Simulations of the invasion of a sexual population, initially gaussian distributed, with a different growth rate at low density]{\textbf{ Simulations of the invasion of a sexual population, initially gaussian distributed, with a different growth rate at low density ($r=0.1$)}, with parameters $\delta t =  0.02$, $\delta x = 4$, $\delta \theta =2 /3$, $x_{\max} = 3000$ and $\theta_{\max} = 201$. (a) Plot of the {\color{Black}population size} $\varrho(t,\cdot)$ for successive fixed times at regular intervals from $t=20$ to $t=200$, with respect to the auto - similar variable $x t^{ - 5/4}$. (b) Plot of the {\color{Black} mean of the dispersive trait} $\overline{\theta}^{num}(t)$  at the front position with respect to time (blue curve) and of the function $t \to {\color{Black}0.85}\,t ^{\color{Black}0.35}$ (red curve), in $\log - \log$ scale. The function $t \to {\color{Black}0.85}\,t ^{\color{Black}0.35}$ is given by a linear regression, with $R^2=0.983$.}
    \label{chap_cemracs_sex:diff_r:fig1}
\end{figure}

\begin{figure}[!h]
    \centering
    \subfloat[$t=50$]{\includegraphics[scale=0.4]{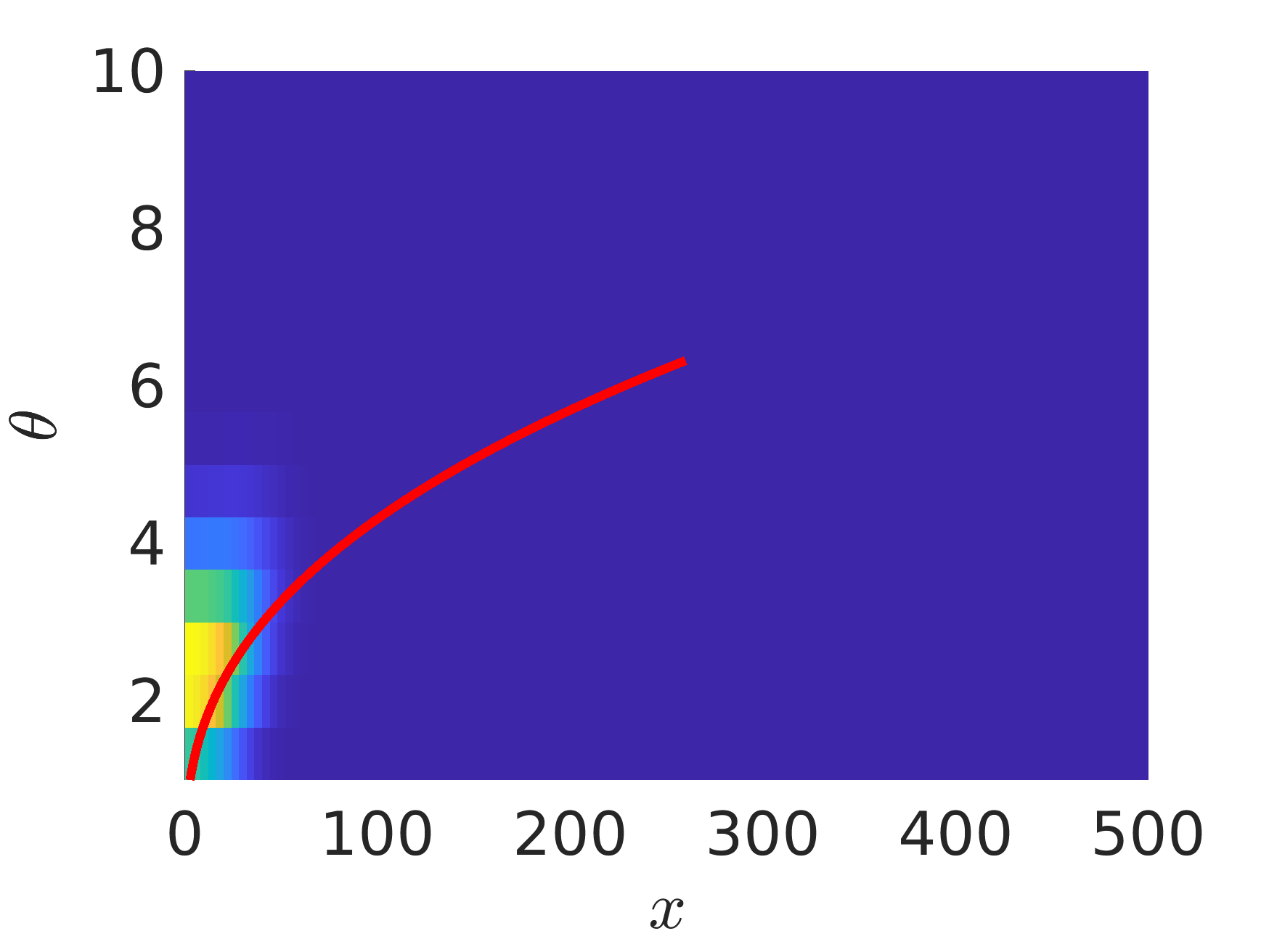}}\hspace{0.3cm}
    \subfloat[$t=100$]{\includegraphics[scale=0.4]{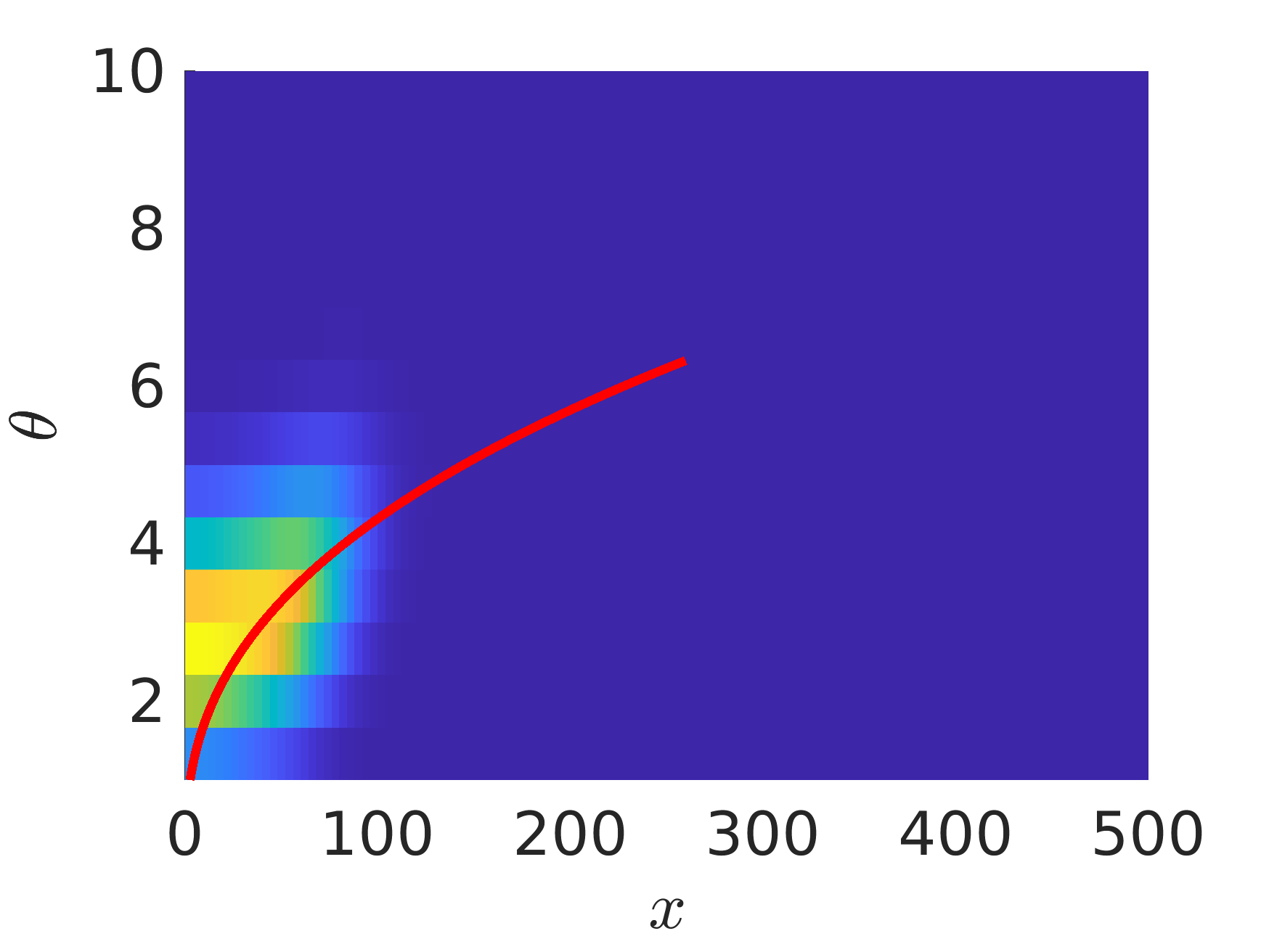}}\\
    \subfloat[$t=150$]{\includegraphics[scale=0.4]{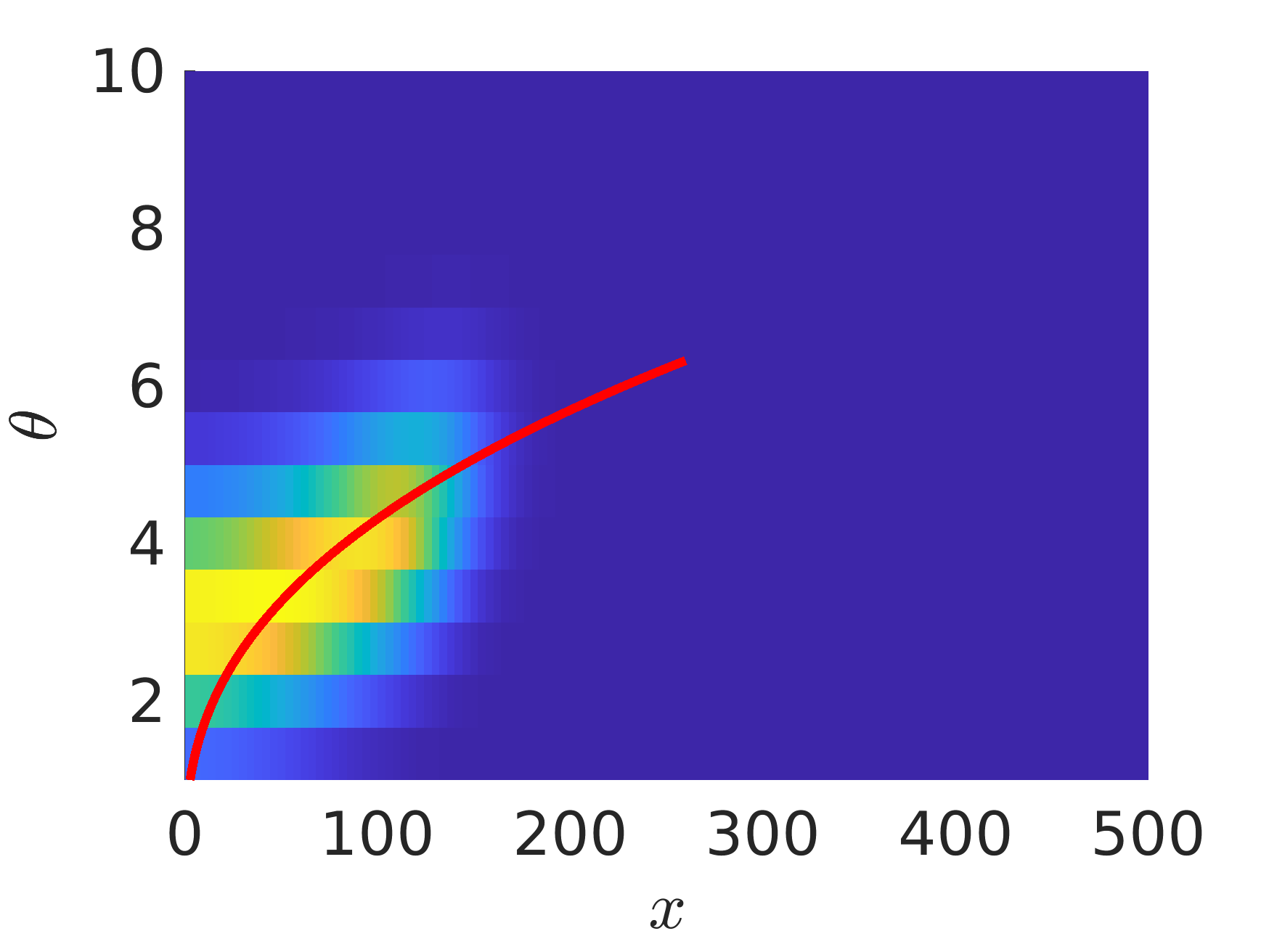}}\hspace{0.3cm}
    \subfloat[$t=200$]{\includegraphics[scale=0.4]{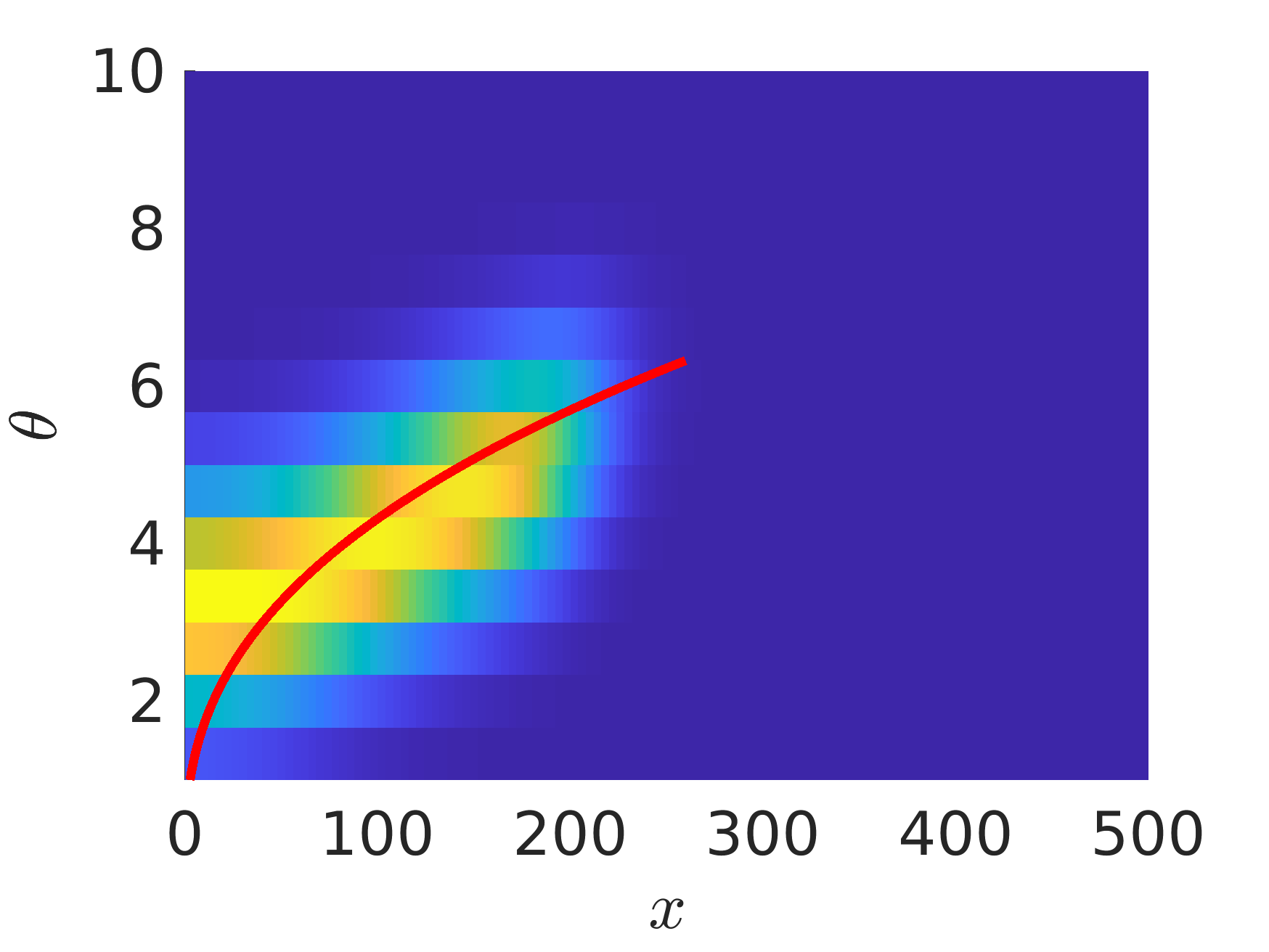}}
    \caption[Contour lines of the trait distribution during the invasion of a sexual population, initially gaussian distributed, with a different growth rate at low density]{\textbf{Contour lines of the trait distribution during the invasion of a sexual population, initially gaussian distributed, with a different growth rate at low density ($r=0.1$)}, given by simulations, at (a) $t=50$ (b) $t=100$ (c) $t=150$ (d) $t=200$. The parameters are $\delta t =  0.02$, $\delta x = 4$, $\delta \theta =2 /3$, $x_{\max} = 3000$ and $\theta_{\max} = 201$.  The red line represents the approximation of the mean  trait behind the {\color{Black}propagating front}.}
    \label{chap_cemracs_sex:diff_r:film}
\end{figure}

\section{Simulations with a different value for \texorpdfstring{$\lambda$}{Lg}} This last supplementary material is devoted to simulations of the invasion of a sexual population, initially gaussian distributed. We change the value of the segregrational variance: we take $\lambda=1$.

With these parameters, we have approximated the position of the {\color{Black}propagating front} at time $t\ge 0$ by:
\[
X(t) = 4\left( \frac{\boldsymbol{\lambda}}  3\right)^{1/2} \, \boldsymbol{r}^{3/4} \ t^{5/4}= {4\over \sqrt 3}\,t^{5/4}\approx 2.31\, t^{5/4}.
\]
This seems again consistent with the simulations (see Fig.~\ref{chap_cemracs_sex:diff_lambda:fig1}). This shows that when the segregational variance $\lambda^2>0$ is strong, the individuals can have individuals with better dispersive trait, which accelerate the invasion. In Fig.~\ref{chap_cemracs_sex:diff_lambda:film}, we see that the approximation given by the associated article is accurate. 

\begin{figure}[!h]
    \centering
    \subfloat[]{\includegraphics[scale=0.65]{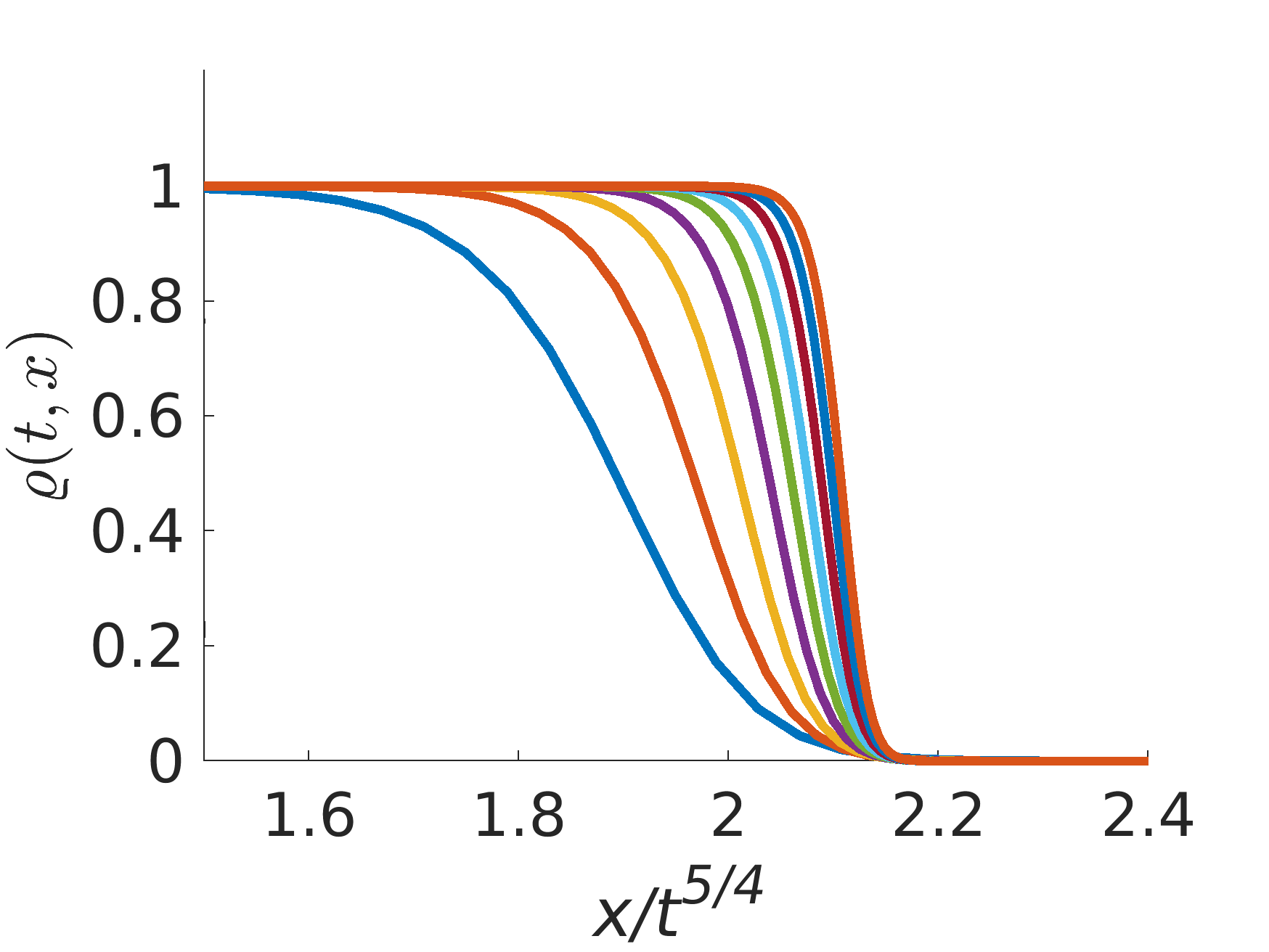}} 
\\    \subfloat[]{\includegraphics[scale=0.65]{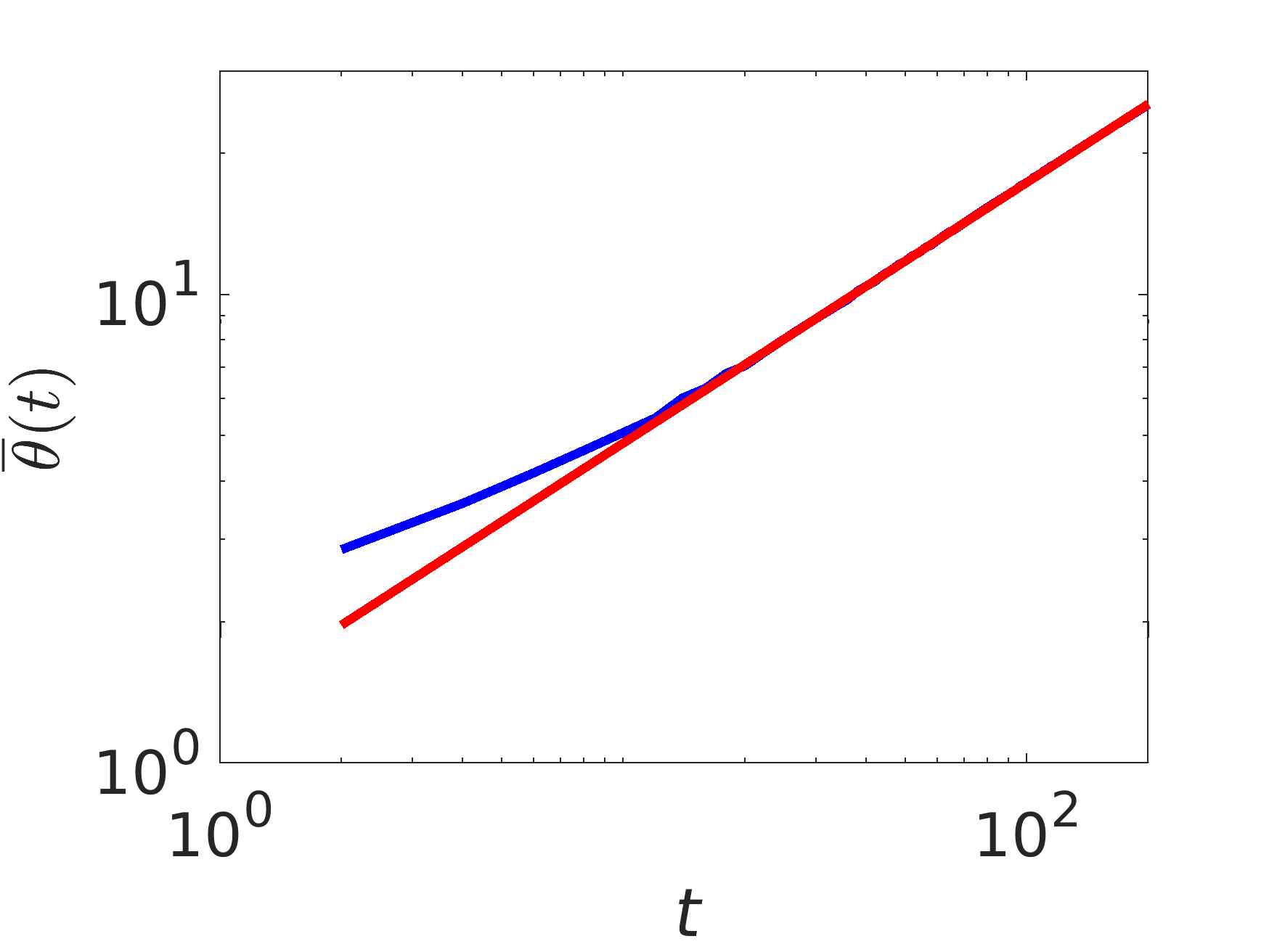}}     \caption[Simulations of the invasion of a sexual population, initially gaussian distributed, with a different segregational variance]{\textbf{ Simulations of the invasion of a sexual population, initially gaussian distributed, with a different segregational variance ($\lambda^2=1$)}, with parameters $\delta t =  0.02$, $\delta x = 4$, $\delta \theta =2 /3$, $x_{\max} = 3000$ and $\theta_{\max} = 201$. (a) Plot of the {\color{Black}population size} $\varrho(t,\cdot)$ for successive fixed times at regular intervals from $t=20$ to $t=200$, with respect to the auto - similar variable $x t^{ - 5/4}$. (b) Plot of the {\color{Black} mean of the dispersive trait} $\overline{\theta}^{num}(t)$  at the front position with respect to time (blue curve) and of the function $t \to  {\color{Black}1.34}\, {t}^{\color{Black}0.56}$ (red curve), in $\log - \log$ scale. The last funtion is given by a linear regression of the {\color{Black} mean of the dispersive trait}, with $R^2=1$.}
    \label{chap_cemracs_sex:diff_lambda:fig1}
\end{figure}

\begin{figure}[!h]
    \centering
    \subfloat[$t=50$]{\includegraphics[scale=0.4]{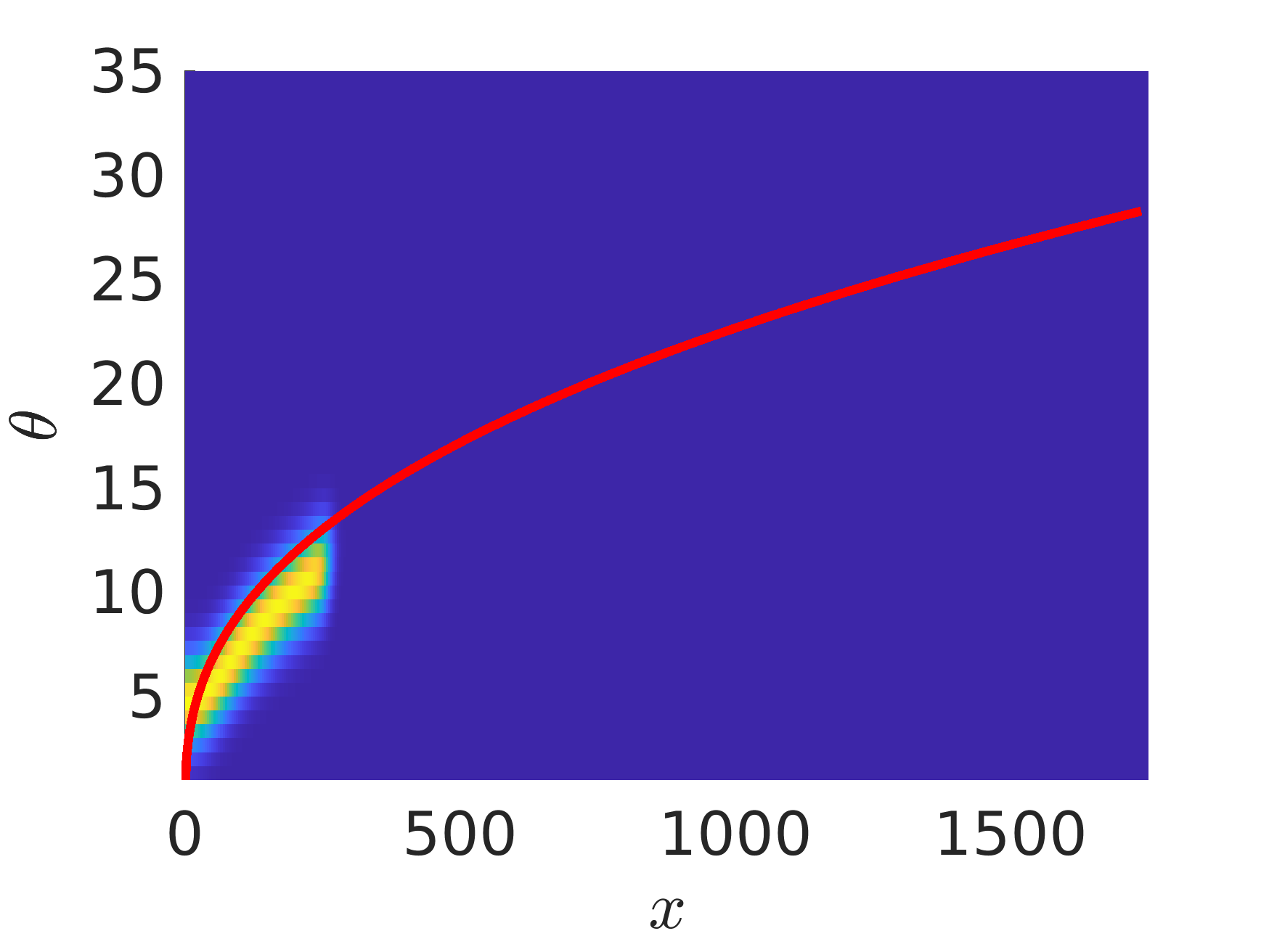}}\hspace{0.3cm}
    \subfloat[$t=100$]{\includegraphics[scale=0.4]{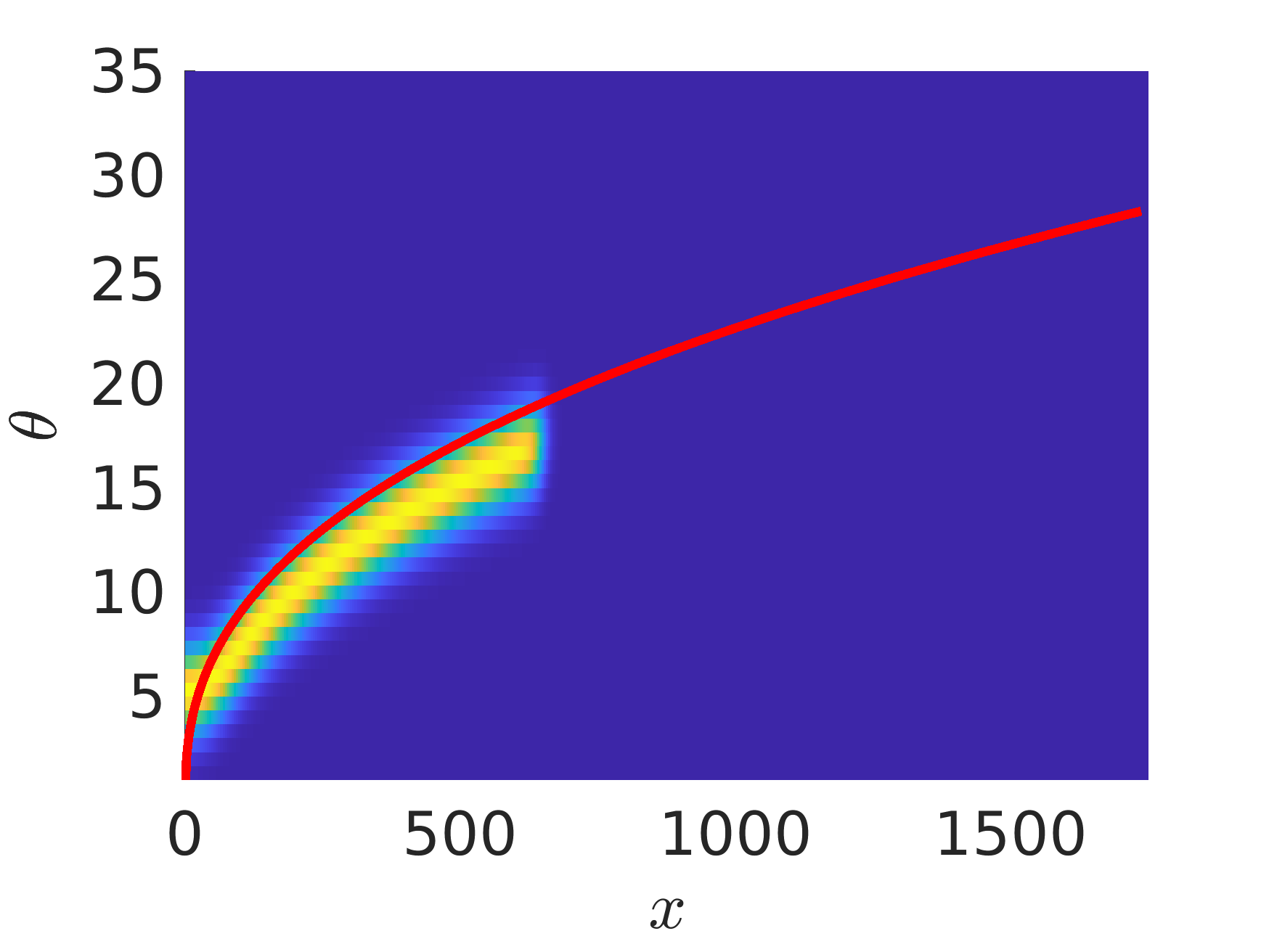}}\\
    \subfloat[$t=150$]{\includegraphics[scale=0.4]{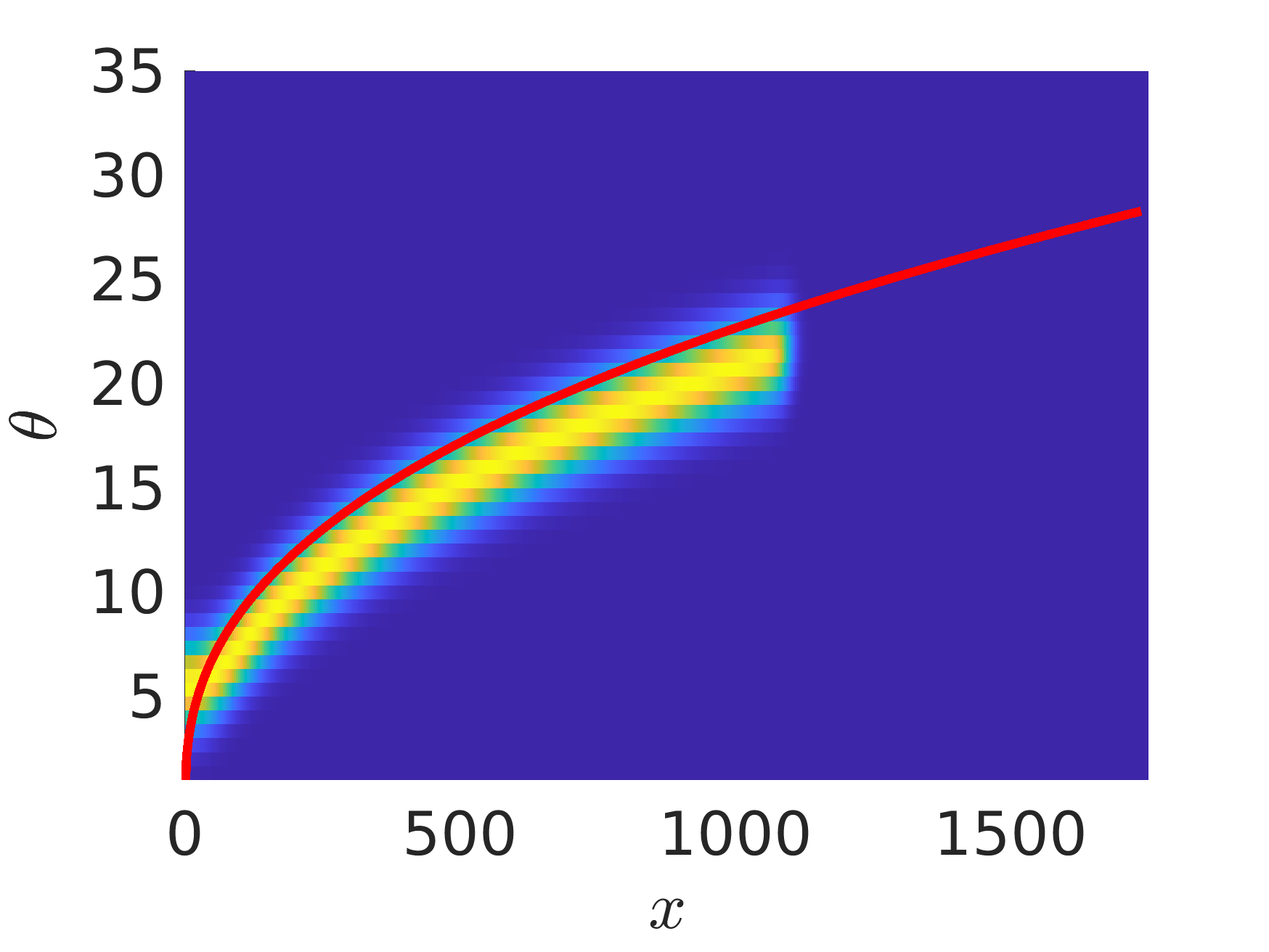}}\hspace{0.3cm}
    \subfloat[$t=200$]{\includegraphics[scale=0.4]{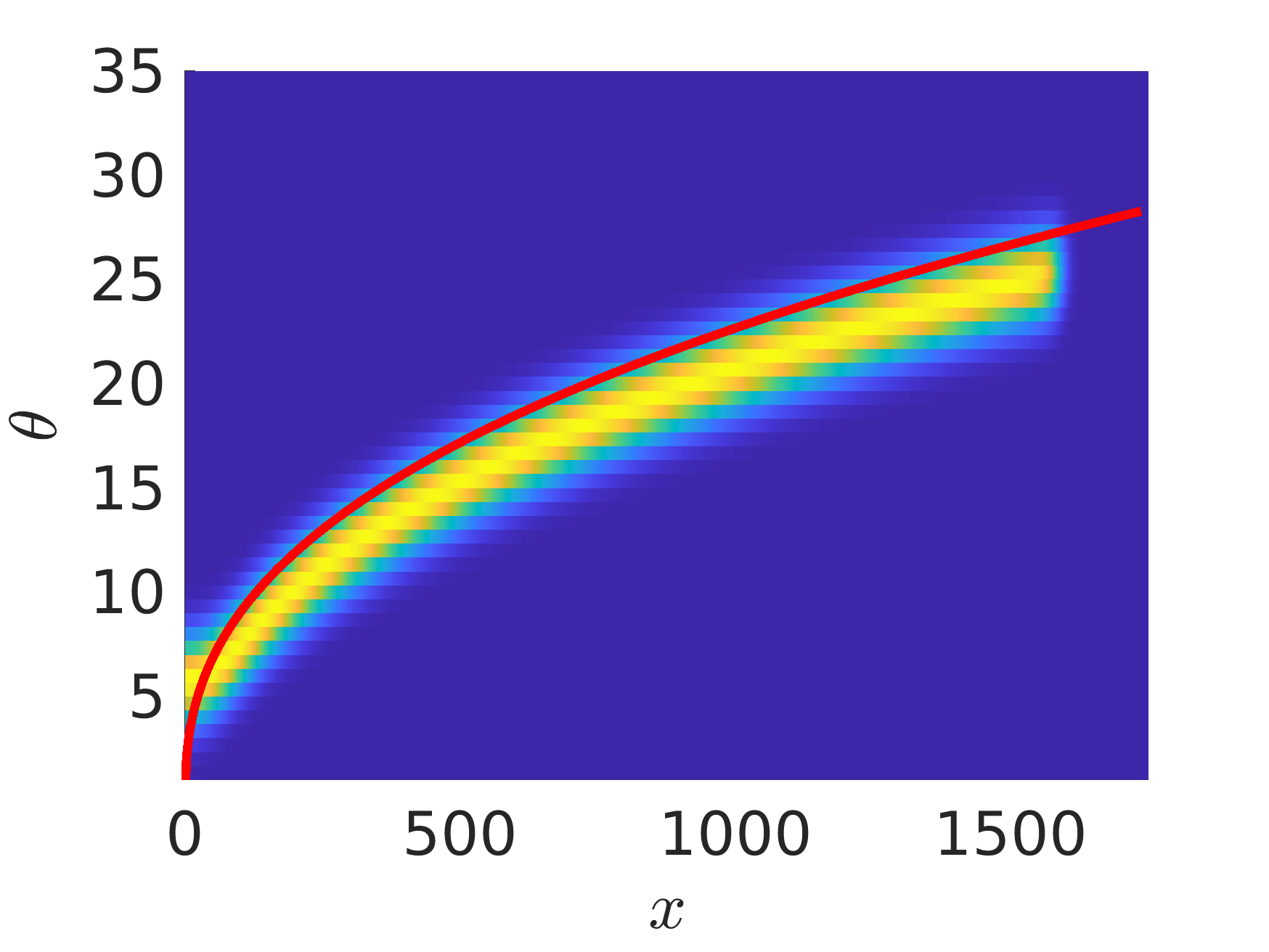}}
    \caption[Contour lines of the trait distribution during the invasion of a sexual population, initially gaussian distributed, with a different segregational variance]{\textbf{Contour lines of the trait distribution during the invasion of a sexual population, initially gaussian distributed, with a different segragional variance ($\lambda^2=1$)}, given by simulations, at (a) $t=50$ (b) $t=100$ (c) $t=150$ (d) $t=200$. The parameters are $\delta t =  0.02$, $\delta x = 4$, $\delta \theta =2 /3$, $x_{\max} = 3000$ and $\theta_{\max} = 201$.  The red line represents the approximation of the mean  trait behind the {\color{Black}propagating front}.}
    \label{chap_cemracs_sex:diff_lambda:film}
\end{figure}